\documentclass[10pt,reqno]{amsart}
\usepackage{amsmath,amssymb,amsfonts,enumerate,amsthm,graphicx,bm}
\usepackage[all]{xy}
\usepackage{amscd}
\input xy
\xyoption{all}
\headsep=1cm \numberwithin{equation}{section}
\newtheorem{Theorem}{Theorem}[section] \newtheorem{Corollary}[Theorem]{Corollary} \newtheorem{Lemma}[Theorem]{Lemma} \newtheorem{Proposition}[Theorem]{Proposition} \theoremstyle{definition}
 \newtheorem{Definition}[Theorem]{Definition} \newtheorem{Example}[Theorem]{Example}  \newtheorem{Remark}[Theorem]{Remark}  \newtheorem{Question}[Theorem]{Question}   \newtheorem{conj}[Theorem]{Conjecture} \numberwithin{equation}{section}
  
\DeclareMathOperator{\Hom}{Hom} 
\DeclareMathOperator{\Image}{Im}
\DeclareMathOperator{\coker}{Coker}
\DeclareMathOperator{\Ker}{Ker}   \DeclareMathOperator{\Ann}{Ann}
 \DeclareMathOperator{\Ass}{Ass}
\DeclareMathOperator{\Proj}{Proj} \DeclareMathOperator{\beg}{indeg}
\DeclareMathOperator{\Deg}{deg} 

 \DeclareMathOperator{\Ext}{Ext}\DeclareMathOperator{\Dim}{dim}
\DeclareMathOperator{\depth}{depth}
 \DeclareMathOperator{\Ht}{ht}
\DeclareMathOperator{\Reg}{reg} \DeclareMathOperator{\Tot}{Tot}

\DeclareMathOperator{\End}{end} \DeclareMathOperator{\END}{end}
 \DeclareMathOperator{\A}{\alpha}
 \DeclareMathOperator{\E}{\textrm{E}}
 
 \DeclareMathOperator{\SD}{\textrm{SD}}
\DeclareMathOperator{\SDC}{\textrm{SDC}} 

\newcommand{\binomial}[2]{{#1 \choose #2}}

\newcommand{\td}{\tilde}

\newcommand{\fm}{\mathfrak{m}}

\newcommand{\fp}{\frak{p}}
\newcommand{\fq}{\frak{q}}
\newcommand{\fa}{\frak{a}}

\newcommand{\ra}{\rightarrow}

\def \y {\otimes}

\def \Z {\mathbb{Z}}
\def \N {\mathbb{N}}

 \def\ff{{\bf f}}\def\aa{{\bf a}}
  
\def\ag{{\bm \gamma}}

\newcommand{\llar}{-\kern-5pt-\kern-5pt\longrightarrow}
\def\restr{{\kern-1pt\restriction\kern-1pt}}

\begin{document}
\title{ Residual Intersections and the annihilator of Koszul homologies }
\author[S.H.Hassanzadeh]{S. H. Hassanzadeh$^1$}
\author[J. Na\'{e}liton]{J. Na\'{e}liton$^2$}
\address{$^{1}$  Instituto de Matem\'{a}tica,
Universidade Federal do Rio de Janeiro,  Brazil. }
\email{hamid@im.ufrj.br and  hassanzadeh.ufrj@gmail.com}

\address{$^{2}$  Instituto de Matem\'{a}tica Pura e Aplicada, IMPA,  Rio de Janeiro,  Brazil. }
 \email{jnaeliton@yahoo.com.br}

\date{\today}
\maketitle
\footnotetext{Mathematics Subject Classification 2010
 (MSC2010). Primary  13C40, 13H10, 13D02,
Secondary 14C17, 14M06.}
\footnotetext{Keywords: Residual Intersection, Canonical Module, Type, Sliding depth, Approximation Complex, Koszul annihilator}
\footnotetext[1]{Corresponding author, Partially
supported by a CNPq grant.}
\footnotetext[2]{Under a CNPq Doctoral scholarship.}

\begin{abstract} Cohen-Macaulayness, unmixedness, the structure of the canonical module and the stability of the Hilbert function of algebraic residual intersections are studied in this paper.   Some  conjectures about these properties  are established for  large classes of residual intersections without restricting  local number of generators of the ideals involved.  A family of approximation complexes for residual intersections is constructed to determine the above properties. Moreover some  general properties of the symmetric powers of quotient ideals are determined  which were not known even for special ideals with  a small number of generators. Acyclicity of a prime case of these complexes  is shown to be  equivalent to  find a common annihilator for higher Koszul homologies. So that, a tight relation between  residual intersections and the uniform annihilator of positive Koszul homologies is unveiled  that sheds some light on their structure.

\end{abstract}
\section{Introduction}
Understanding the residual of the intersection of two algebraic varieties is an intriguing concept in algebraic geometry. Similar to many other notions in intersection theory,  the concept of residual intersections need the proper interpretation to encompass the desired algebraic and geometric properties. In the sense of  Artin and Nagata \cite{AN},  which is our point of view, the notion is  a vast generalization of linkage (or liaison) which  is more ubiquitous, but also harder to understand. Let $X$ and $Y$ be  irreducible
closed subschemes of a Noetherian scheme $Z$ with ${\rm codim}_Z(X)\leq {\rm codim}_Z(Y)=s$ and $Y \not \subset X$,
 Then $Y$ is called a residual intersection of $X$ if the number of equations needed to define $X\cup Y$ as a subscheme
of $Z$ is the smallest possible that is  $s$. Precisely,
if $R$ is a  Noetherian  ring,  $I$  an  ideal of
height $g$  and $s \geq g$  an integer, then
\begin{itemize}
\item{An {\it (algebraic) $s$-residual intersection} of $I$ is a proper ideal $J$ of $R$ such that $\Ht(J)\geq s$ and $J=(\fa:_R I)$ for some ideal $\fa\subset I$  generated by $s$ elements.}
\item{A {\it geometric $s$-residual intersection} of $I$ is an algebraic $s$-residual intersection $J$ of $I$ such that $\Ht(I+J)\geq s+1.$}
\end{itemize}

Based on a construction of Laksov for residual intersection, Fulton \cite[Definition 9.2.2]{Fu} presents a formulation  for residual intersection that, locally, can be expressed as follows:
Suppose that $X=Spec(R)$ and $Y$ and $S$ are  closed subschemes of $X$ defined by the ideals $\fa$ and $I$ respectively. Let $\widetilde{X}=\Proj(\mathcal{R}_R(I))$  be the blow-up of $X$ along $S$.
Consider the natural map $\pi: \widetilde{X}\to X$. Let $\widetilde{Y}=\pi^{-1}(Y)$ and  $\widetilde{S}=\pi^{-1}(S)$. Then $\widetilde{Y}=\Proj(\mathcal{R}_R(I)/\fa\mathcal{R}_R(I))$ and
$\widetilde{S}=\Proj(\mathcal{R}_R(I)/I\mathcal{R}_R(I))$ are closed subschemes of $\widetilde{X}$ with ideal sheaves $\mathcal{I}_{\widetilde{Y}}$ and $\mathcal{I}_{\widetilde{S}}$ .  $\mathcal{I}_{\widetilde{S}}$ is an invertible sheaf. Let $Z'\subseteq \widetilde{X}$ be the closed subscheme defined by the ideal sheaf $\mathcal{I}_{Z'}=\mathcal{I}_{\widetilde{Y}}\cdot \mathcal{I}_{\widetilde{S}}^{-1}$. $Z'$ is called the {\it residual scheme to $\widetilde{S}$ in $\widetilde{Y}$}. Preciesly $Z'=\Proj(\mathcal{R}_R(I)/\gamma\mathcal{R}_R(I))$  in which $\gamma=\fa\subseteq (\mathcal{R}_R(I))_{[1]}$.
Finally, the  residual intersection to $S$ in $Y$ \cite[9.2.2]{Fu} is the direct image of $\mathcal{O}_{Z'}$ i.e. $\pi_{\ast}(\mathcal{O}_{Z'})$. Since $\mathcal{O}_{Z'}$ is a coherent sheaf, by \cite[III, Proposition 8.5]{H}, $$\pi_{\ast}(\mathcal{O}_{Z'})=\widetilde{H^0(\widetilde{X},\mathcal{O}_{Z'})}=\widetilde{\Gamma(\widetilde{X},\mathcal{O}_{Z'})}.$$
Note that $\Gamma(\widetilde{X},\mathcal{O}_{Z'})=\Gamma_{\ast}(\widetilde{X},\mathcal{O}_{Z'}(0))$ which is equal to the ideal transform
$D_{\mathcal{R}_R(I)_+}(\mathcal{R}_R(I)/\gamma\mathcal{R}_R(I))_{[0]}, $ by an application of \u{C}ech complex. The later is closely related, and in many cases determined, by the ideal
 $$H^0_{\mathcal{R}_R(I)_+}(\mathcal{R}_R(I)/\gamma\mathcal{R}_R(I))_{[0]}=\bigcup (\fa I^i:_RI^{i+1}).$$

  In \cite{Ha} and  in the current work, we  consider the symmetric algebra $Sym_R(I)$ instead of the Rees algebra $\mathcal{R}_R(I)$. This generalization  has already proved its usefulness in studying multiple-point formulas by Kleiman c.f. \cite[17.6.2]{Fu}. We then find a kind of {\it Arithmetical} residual intersection to be $\bigcup (\gamma Sym_R^i(I):_RSym_R^{i+1}(I))$ where $\gamma=\fa\subseteq Sym_R^i(I)_{[1]}$.  Comparing the above three definitions for residual intersection,  we have
$$J=(\fa:I)\subseteq  \bigcup (\gamma Sym_R^i(I):_RSym_R^{i+1}(I)) \subseteq  \bigcup (\fa I^i:_RI^{i+1}).$$

Interestingly, these ideals coincide if the algebraic residual intersection $J$  does not share any associated primes with $I$, e.g if $J$ is unmixed and the residual is geometric.

Determining the cases where the first inclusion above is an equality leads us to define a third variation of algebraic residual intersection.
\begin{Definition}\label{darithmetic} An {\it arithmetic $s$-residual intersection}  $J=(\fa:I)$ is an algebraic $s$-residual intersection such that $\mu_{R_{\fp}}((I/\fa)_{\fp})\leq 1$ for all prime ideal $\fp \supseteq (I+J)$ with $\Ht(\fp)\leq s$. ( $\mu$ denotes the minimum number of generators.)
\end{Definition}

Clearly any geometric $s$-residual intersection is arithmetic. Moreover for any
algebraic $s$-residual intersection $J=(\fa:(f_1,\cdots,f_r))$ which is not geometric, all of the colon ideals $(\fa:f_i)$ are arithmetic $s$-residual
intersection and at least one of them is not geometric.

In this paper  we introduce a family  of complexes, denoted by $\{_i\mathcal{Z}^{+}_{\bullet}\}_{i=0}^{\infty}$, to approximate $i$th symmetric power of $I/\fa$, $Sym_R^i(I/\fa)$ for $i>1$. The idea to define this family is inspired by \cite{Ha} in which the single complex  $_0\mathcal{Z}^{+}_{\bullet}$ is treated.  $H_0(_0\mathcal{Z}^{+}_{\bullet})$  is a cyclic module of the form $R/K$ where $K$ is called  the {\it Disguised $s$-residual intersection of $I$ w.r.t. $\fa$}, see Definition \ref{DDisguised}.  The study of the other members of the above family of complexes shed some  more light on the structure of residual intersections.  A flavor of  our main results in  Section \ref{SecZ+} is the following,

 {\it Let $(R,\fm)$ be a  CM local ring of dimension $d$ and $I$ be an ideal with $\Ht(I)=g >0$. Let $s\geq g$, $1\leq k\leq s-g+2$ and $J=(\fa:I)$  be any (algebraic) $s$-residual intersection. Suppose that $I$ is Strongly Cohen-Macaulay . Then
\begin{enumerate}[(1)]
\item{\text{Corollary \ref{ccanonical},} the canonical module of $R/J$ is $Sym_R^{s-g+1}(I/\fa)$, provided the residual is arithmetic and $R$ is Gorenstein;}
\item{\text{Corollary \ref{cunmixed}}, $\depth(R/\fa)=d-s$; }
\item{ \text{Corollary \ref{cunmixed}}, $J$ is  unmixed  of codimension $s$;}
\item{ \text{Theorem \ref{t2} and Proposition \ref{phil}}, $_k\mathcal{Z}^{+}_{\bullet}$ is acyclic,
$H_0(_k\mathcal{Z}^{+}_{\bullet})=Sym_R^{k}(I/\fa)$ and the latter is  CM  of dimension $d-s$;
 (the acyclicity of  $_k\mathcal{Z}^{+}_{\bullet}$ implies conjecture(5) below  in the arithmetic case.)}

\end{enumerate}
}

These results address the following (implicit) conjectures made during the development of the theory of algebraic residual intersections.

{\it \label{conj}Let $R$ be a Cohen-Macaulay(CM) local ring and  $I$ be  Strongly Cohen-Macaulay (SCM) or even just satisfies Sliding depth (SD). Then
\begin{enumerate}
\item{$R/J$ is  Cohen-Macaulay.}
\item{The canonical module of $R/J$ is the $(s-g+1)$th symmetric power of $I/\fa$, if $R$ is Gorenstein.}
\item{$\fa$ is minimally generated by $s$ elements.}
\item{ $J$ is unmixed.}
\item{The Hilbert series of $R/J$ depends only on $I$ and the degrees of the generators of $\fa$.}
\end{enumerate}
}
The first conjecture essentially goes back to \cite{AN}, it has been asked as an open question in \cite{HU}. However in the  Sundance conference in 1990 \cite{U1}, Ulrich  mentioned the above conjectures $(1)-(4)$ as desirable facts to be  proved.
The  property of the Hilbert function is rather recent and has been analyzed by Chardin, Eisenbud and Ulrich in \cite{CEU}.

It should be mentioned that these conjectures are proved if one supposes  in addition that the ideal $I$   has locally few number of generators, a condition which is called $G_s$, or if it has a deformation with $G_s$ property. Over time the $G_s$ condition became a 'standard' assumption in the theory of residual intersection which is not avoidable in some cases. However, the desire is to prove the above assertions without restricting the  local number of generators of $I$.

In comparison, obtaining the structure of the canonical module in the
absence of the $G_s$ condition  is more challenging. To achieve this, we show that under the above mentioned hypotheses, $Sym_R^{s-g+1}(I/\fa)$ is a faithful maximal Cohen-Macaulay
 $R/J$-module of type $1$. Concerning the type of modules, we prove even more. We show in Theorem \ref{ttype} that the following  inequality holds for any
 $1\leq k\leq s-g+1$, $$r_R(Sym_R^{k}(I/\fa))\leq \binom{r+s-g-k}{r-1}r_R(R),$$
where for a finitely generated $R$-module $M$, $r_R(M):=\Dim_{R/\fm}\Ext^{\depth(M)}_R(R/\fm,M)$ is the Cohen-Macaulay type of  $M$.

 In section \ref{App}, we present several applications of the theorems and constructions so far. We state how much the Hilbert functions of $R/J$ and $R/\fa$ depend on the generators and/or degrees of $I$ and $\fa$.  If $I$ satisfies SCM, then the Hilbert function of  the disguised residual intersection,  and that of $Sym_R^k(I/\fa)$, if $1\leq k\leq s-g+2$, depends only on the \textbf{degrees} of the generators of $\fa$ and the Koszul homologies of $I$.
In particular,  $k=1$ implies that  the Hilbert function of $R/\fa$ is constant on the open set of ideals $\fa$ generated by $s$
forms of the given degrees such that $\Ht(\fa : I)\geq s$. This is comparable with results in \cite{CEU} where the same assertion is concluded under some $G_s$ hypotheses.

The graded structure of $_k\mathcal{Z}^{+}_{\bullet}$ shows that
if $I$ satisfies  the SD$_1$ condition, then
 \begin{center}
 $\Reg (Sym_R^k(I/\fa)) \leq \Reg( R) + \Dim(R_0)+\sigma( \fa)-(s-g+1-k)\beg(I/\fa)-s $.
\end{center}
for $k\geq 1$. These applications were no known even if  ideal $I$ satisfies the $G_s$ condition. Finally in this section, Proposition \ref{PGs}, by a combination
of older and newer facts, we show that for any
algebraic $s$-residual intersection $J=\fa:I$, if $I$ is SCM and evenly linked to a $G_s$ ideal (or has a deformation with these properties)  then the disguised residual intersection and the algebraic residual intersection coincide. Based on this fact, we conjecture that, Conjecture \ref{Conj1},  {\it in the presence of the sliding depth condition the disguised residual intersection is the same as the algebraic residual intersection}. 

In section \ref{Arith}, we try to understand  better the structure of complex $_0\mathcal{Z}^{+}_{\bullet}$ in the case where $I/\fa$ is principal, say $I=(\fa,b)$. We find in Theorem \ref{thr2} that $H_i(_0\mathcal{Z}^{+}_{\bullet}(\aa,\ff))\simeq bH_{i}(a_1,\cdots,a_s)$ for all $i\geq 1$ and $H_0(_0\mathcal{Z}^{+}_{\bullet}(\aa,\ff))\simeq R/(\fa:b)$. This fact shows how much the homologies of $\mathcal{Z}^+_{\bullet}$ complexes may depend on the generating sets, as well it shows a tight relation between uniform annihilator of Koszul homologies and acyclicity of $\mathcal{Z}^+_{\bullet}$. As a byproduct, Corollary \ref{Ckosan} suggests that in the way of studying the properties of colon ideals, instead of assuming $\Ht((a_1,\cdots,a_s):I)\ge s$, one may only need to suppose that  $IH_i(a_1,\cdots,a_s)=0$ locally at codimension $s-1$.

Motivated by the facts in section \ref{Arith}, we  investigate  on  the uniform annihilator of positive  Koszul homologies in  Section \ref{sann}. Not much is known about the annihilator of Koszul homologies.
In Corollary \ref{creskosan}, we show that for a residual intersection $J=\fa:I$ where $I$ satisfies SD, and $\depth(R/I)\geq d-s$,   $$I\subseteq \bigcap_{\substack{j\geq 1}}\Ann(H_j(\fa)).$$

Surprisingly, this results contradicts one of the unpublished, but well-known results of G. Levin \cite[Theorem 5.26]{V2} which yielded  in  \cite{CHKV} that Supp$(H_1(\fa))$=Supp$(H_0(\fa))$. Simple examples of residual intersection disprove this last claim. Moreover in this section it is shown in Theorem \ref{tsd} that for an $s$-residual intersection $J=(\fa:I)$, if $I$ satisfies $\SD$ and $\depth (R/I)\geq d-s$,
 so does  $\fa$. This is an interesting result since for a long time it was  known that the residual intersections of the  ideal $\fa$ are Cohen-Macaulay although no one was aware of the $\SD$ property of $\fa$. 
\section{Residual Approximation Complexes}\label{SecZ+}
In this section we introduce a family of complexes which approximate the residual intersection and some of its related symmetric powers. We denote this family by $\{_i\mathcal{Z}^{+}_{\bullet}\}_{i=0}^{\infty}$. The complex $_0\mathcal{Z}^{+}_{\bullet}$ was  already defined in \cite{Ha} and used to prove the CM-ness of arithmetic  residual intersection of ideals with sliding depth.


Throughout this section, $R$ is a Noetherian ring of dimension
$d$, $I=(\ff)=(f_1,\cdots,f_r)$ is an ideal of grade $g\geq1$. Although by adding one variable we may as well treat the case $g=0$;
 but, for simplicity, we keep the assumption $g\geq1$.
$\fa=(a_1,\cdots,a_s)$ is an ideal contained in $I$, $s\geq g$, $
J=\fa:_{R}I$, and $S=R[T_1,\cdots,T_r]$ is a polynomial extension
of $R$ with indeterminates $T_i$'s. We denote the symmetric
algebra of $I$ over $R$ by $\mathcal{S}_I$ or in general the symmetric algebra of an $R$-module $M$ by $Sym_R(M)$ and the $k$th symmetric power of $M$ by $Sym_R^k(M)$.  We consider
$\mathcal{S}_I$ as an $S$-algebra via the ring homomorphism
$S\rightarrow \mathcal{S}_I$ sending $T_i$ to  $f_i$ as an
element of  $(\mathcal{S}_I)_1=I$ then $\mathcal{S}_{I}=S/\mathcal{L}$.
 Let $a_i= \sum _{j=1}^{r}c_{ji}f_j$,  $\gamma_i= \sum _{j=1}^{r}c_{ji}T_j$, $\bm{\gamma}=(\gamma_1,\cdots,\gamma_s)$ and
$\frak{g}:=(T_1,\cdots,T_r)$.

For a sequence of elements $\frak{x}$
in a commutative ring $A$ and an $A$-module $M$, we denote the
Koszul complex by $K_\bullet(\frak{x};M)$, its cycles by
$Z_i(\frak{x};M)$ and homologies by $H_i(\frak{x};M)$. For a graded
module $M$, $\beg(M):=\inf\{i : M_i\neq 0\}$ and
$\End(M):=\sup\{i : M_i\neq 0\}$. Setting $\deg(T_i)=1$ for all
$i$, $S$ is a  standard graded ring over $S_0=R$.

To set one more convention, when we draw the picture of a double
complex obtained from a tensor product of two finite complexes
(in the sense of \cite[2.7.1]{W}), say $\mathcal{A}\bigotimes
\mathcal{B}$; we always put $\mathcal{A}$ in the vertical
direction and $\mathcal{B}$ in the horizontal one. We also label
the module which is  in the up-right corner by $(0,0)$ and
consider the labels for the rest, as the points in the
third-quadrant.

\subsection{$_k\mathcal{Z}^{+}_{\bullet}$ complexes}

The first object  in the construction of the family of  $_k\mathcal{Z}^{+}_{\bullet}$ complexes is  one of the approximation complexes- the $\mathcal{Z}$-complex \cite{HSV}. We consider  the approximation complex  $\mathcal{Z_{\bullet}}(\ff)$: $$0\ra Z_{r-1}\y_R S(1-r)\ra\cdots\ra Z_{1}\y_R S(-1)\ra Z_{0}\y_R S\ra 0$$
 where $Z_i=Z_i(\ff)$ is the $i$th cycle of the Koszul complex $K_\bullet(\ff,R)$.

 The second  object is the Koszul complex
 $K_{\bullet}(\bm{\gamma},S)$:

$$
0\ra K_s(\gamma_1,...,\gamma_s)(-s)\ra \cdots \ra K_{1}(\gamma_1,...,\gamma_s)(-1)\ra K_0(\gamma_1,...,\gamma_s)\ra 0.$$
Let   $\mathcal{D_{\bullet}}={\rm Tot}(K_{\bullet}(\bm{\gamma},S)\y_S \mathcal{Z}_{\bullet}(\ff))$. Then

\begin{equation}\label{Di}\mathcal{D}_i=\bigoplus_{j=i-s}^{\min\{i,r-1\}}[Z_{j}\y_R S]^{\binomial{s}{i-j}}(-i).
 \end{equation}

For a graded $S$-module $M$ the $k$th graded component of $M$ is denoted by $M_{[k]}$. Let $(D_{\bullet})_{[k]}$ for ${k\in\Z}$ be the $k$th graded strand of $D_{\bullet}$. We have $(D_i)_{[k]}=0,$ for all $k<i$, in particular
\begin{equation}\label{kerD}
H_{k}((D_{\bullet})_{[k]})=\Ker(D_k \rightarrow D_{k-1})_{[k]}.
\end{equation}

Now, let $C_{\frak{g}}^{\bullet}=C^{\bullet}_{\frak{g}}(S)$be   the \v{C}ech complex of $S$ with respect to the sequence  $\frak{g}=(T_1,\cdots,T_{r})$, 

$$C_{\frak{g}}^{\bullet}:~~0\ra C^0_{\frak{g}}(=S)\ra C^1_{\frak{g}}\ra \cdots \ra C^r_{\frak{g}}\ra 0. $$

 We then consider the bi-complex $C^{\bullet}_{\frak{g}}\y_S D_{\bullet}$ with $C^{0}_{\frak{g}}\y_S D_{0}$ in  the corner. This bi-complex    gives rise  to  two spectral sequences for which the second terms of the horizonal spectral are

\begin{equation}\label{E2Hor}
^{2}\E _{hor}^{-i,-j}=H^{j}_{\frak{g}}(H_{i}(D_{\bullet})),
\end{equation}
 and  the first terms of the vertical spectral are
\begin{equation}\label{Ever1}\E_{ver}^{-i,-j}=\left \{
\begin{array}{ccccccc}
0\ra &H^{r}_{\frak{g}}(D_{r+s-1})\ra &\cdots\ra &H^{r}_{\frak{g}}(D_{1})\ra &H^{r}_{\frak{g}}(D_{0})\ra &0& \text{if~~} j=r \\
 0& \text{otherwise}
\end{array}
\right. \end{equation}

Since $H^{r}_{\frak{g}}(D_{i})=H^{r}_{\frak{g}}(\bigoplus_{j}[Z_{j}\y S](-i))=\bigoplus_{j}Z_{j}\y H^{r}_{\frak{g}}(S)(-i)$ and $\End (H^{r}_{\frak{g}}(S))=-r$, it follows that $\End(H^{r}_{\frak{g}}(D_{i}))=i-r$, thus $H^{r}_{\frak{g}}(D_{i})_{[i-r+j]}=0$, for all $j\geq 1$. It then leads to  define the following sequence of complexes indexed by  $k\geq 0$:

\begin{eqnarray}\label{complex1}
\xymatrix{0\ar[r]&H^{r}_{\frak{g}}(D_{r+s-1})_{[k]}\ar[r]&...\ar[r]&H^{r}_{\frak{g}}(D_{r+k+1})_{[k]}\ar[r]^{\phi_k}&H^{r}_{\frak{g}}(D_{r+k})_{[k]}\ar[r]&0}
\end{eqnarray}

Since the vertical spectral (\ref{Ever1}) collapses at the second step,
 the horizonal spectral converges to the homologies of $H^{r}_{\frak{g}}(D_{\bullet})$.
Since all of the homomorphisms are homogeneous of degree $0$,
the convergence inherits to any graded component.
 Therefore for any $k\geq 0$ there exists  a filtration $\cdots \subseteq \mathcal{F}_{2k}\subseteq \mathcal{F}_{1k}\subseteq \coker(\phi_k) $ such that
\begin{equation}\label{F1}
\frac{\coker(\phi_{k})}{\mathcal{F}_{1k}}\simeq(^{\infty}\E_{hor}^{-k,0})_{[k]}.
 \end{equation}
 Observing that $H^{-t}_{\frak{g}}(H_h(D_{\bullet}))=0$, for all $(t,h)\in \N \times \N_0$, one has $^{l}E_{hor}^{-k,0}\subseteq H^{0}_{\frak{g}}(H_k(D_{\bullet}))$, for all $l\geq 2$.  Hence we have the following chain of maps for which, except the isomorphism in the middle,  all  maps are canonical and  the map on the right is given by (\ref{kerD}):

\begin{eqnarray}\label{tau}
\xymatrix{ H^{r}_{\frak{g}}(D_{r+k})_{[k]}\ar[r]& \coker(\phi_k)\ar[r]&\frac{\coker(\phi_{k})}{\mathcal{F}_{1k}}\ar[lld]_{\simeq}&  \\ ^{\infty}(\E_{hor}^{-k,0})_{[k]}\ar[r]^{1-1} & H^{0}_{\frak{g}}(H_k(D))_{[k]}\ar[r]^{1-1} &H_k(D)_{[k]} \ar[r]^{1-1}& (D_k)_{[k]}.}
 \end{eqnarray}

We denote the composition of the above chain of $R$-homomorphisms by $\tau_k$.

Finally, we define the promised family of complexes as follows:
For any integer $k\ge 0$, $_k\mathcal{Z}_{\bullet}^{+}$ is a complex of length $s$ consists of two parts:
the right part  is $(D_{\bullet})_{[k]}$ and the left part is $(^{1}\E_{ver})_{[k]}$. These parts are joined via $\tau_k$. More precisely,
\begin{eqnarray}\label{kZ+}
\xymatrix{_k\mathcal{Z}_{\bullet}^{+}:&0\ar[r]&_k\mathcal{Z}_{s}^{+}\ar[r]&\cdots\ar^{\phi_k}[r]&_k\mathcal{Z}_{k+1}^{+}\ar[r]^{\tau_k}&_k\mathcal{Z}_{k}^{+}\ar[r]&\cdots\ar[r]&_k\mathcal{Z}_{0}^{+}\ar[r]&0}
\end{eqnarray}
wherein
\begin{equation}\label{z+}
_k\mathcal{Z}_i^{+}=\left \{
\begin{array}{cc}
(D_i)_{[k]}& i\leq \min \{k,s\}, \\
 H^{r}_{\frak{g}}(D_{i+r-1})_{[k]}& i>k.

\end{array}
\right. \end{equation}\\


The structure of $_k\mathcal{Z}^{+}_{\bullet}$ depends in two ways  on the generating sets. Namely,  it depends on the
 generating set of $I$, $\ff$, and the expression of the generators of $\fa$ in terms of the generators of $I$, $c_{ij}$.
However, for $k\geq 1$, we have  $$H_0(_k\mathcal{Z}^{+}_{\bullet})=H_0(D_{\bullet})_{[k]}=(\mathcal{S}_{I}/(\gamma)\mathcal{S}_{I})_{[k]}=Sym_R^k(I/\fa).$$
The case where $k=0$ is also very interesting. However the structure of $H_0(_0\mathcal{Z}^{+}_{\bullet})$ is not as clear as the cases where $k>0$.
\begin{Definition}\label{DDisguised} Let $R$ be a Noeherian ring and $\fa\subseteq I$ be two ideals of $R$.  The {\it Disguised $s$-residual intersection of $I$ w.r.t. $\fa$} is the unique ideal $K$ such that $H_0(_0\mathcal{Z}^{+}_{\bullet})=R/K$.  The reasons for choosing  the attribute {\it Disguised} is that, $K$ is contained in $J=(\fa:I)$, it has the same radical as $J$ has; and in the case where $R$ is CM, $J$ is an algebraic reisudal intersection  and $I$ satisfies some sliding depth condition,  $K$ is Cohen-Macaulay; moreover $K$ coincides with $J$ in the the case where the residual is arithmetic, by \cite[Theorem 2.11]{Ha}. We conjecture, \ref{Conj1}, that  under the above assumptions, $K$ is always the same as $J$ despite it does not appear so. 
\end{Definition}

\subsection{Acyclicity and Cohen-Macaulayness}
One more occasion where the properties of $_k\mathcal{Z}^{+}_{\bullet}$ are independent of the generating sets of $I$ and $\fa$ is  the following lemma, which is crucial to prove the acyclicity of $_k\mathcal{Z}^{+}_{\bullet}$.

\begin{Lemma} \label{l1}  Let $R$ be a Noetherian ring and $\fa\subseteq I$ be two ideals of $R$.  If $I=\fa$, then every complex
 $_k\mathcal{Z}^{+}_{\bullet}$ defined in (\ref{kZ+}) is exact.
\end{Lemma}

\begin{proof}

Since the approximation complex $\mathcal{Z}_{\bullet}(\ff)$ (resp. $K_{\bullet}(\bm{\gamma})$) is a differential graded algebra,
$H_i(\mathcal{Z}_{\bullet})$ (resp. $H_i(K_{\bullet}(\bm{\gamma}))$) is a $\mathcal{S}_{I}=S/\mathcal{L}$-module (resp. $S/(\bm{\gamma})$-module) for all $i$.
The bi-complex $K_{\bullet}(\bm{\gamma})\y_S \mathcal{Z}_{\bullet}$, gives rise to the horizontal spectral sequence with second terms
$^2\E^{-i,-j}_{hor}=H_j(K_{\bullet}(\bm{\gamma};H_i(\mathcal{Z}_{\bullet})))$. It follows that $(\mathcal{L}+(\bm{\gamma}))$
 annihilates $^2\E^{-i,-j}_{hor}$ and consequently annihilates $^{\infty}\E^{-i,-j}_{hor}$ which is a sub-quotient of $^2\E^{-i,-j}_{hor}$.
By the convergence of the spectral sequence to the homologies of the total complex $H_{\bullet}(D_{\bullet})$,
it is straightforward to deduce that $H_{i}(D_{\bullet})$ is $(\mathcal{L}+(\bm{\gamma}))$- torsion modules for all $i$,
i.e.  $(\mathcal{L}+(\bm{\gamma}))^{N}H_{i+j}(D_{\bullet})=0$ for all $i,j$ and for some $N$. Considering the equation $I=\fa$ in $(\mathcal{S}_{I})_1$, we have  $\mathcal{L}+(\bm{\gamma})=\mathcal{L}+\frak{g}$. Therefore
\\
\begin{equation}\label{Ehor}H^{j}_{\frak{g}}(H_i(D_{\bullet}))=H^{j}_{\frak{g}+\mathcal{L}}(H_i(D_{\bullet}))=H^{j}_{(\bm{\gamma})+\mathcal{L}}(H_i(D_{\bullet}))=\left \{
\begin{array}{cc}
H_i(D_{\bullet}),& j=0 \\
 0& j>0.
\end{array}
\right.
\end{equation}\\
Once more we study the spectral sequences arising from $C^{\bullet}_{\frak{g}}\y D_{\bullet}$ as it was already mentioned in (\ref{kZ+}). Applying (\ref{Ehor}), we draw both horizontal and vertical spectral sequence simultanously as follows:

 \begin{eqnarray*}
\xymatrix{
^{\infty}\E_{hor}:&0&\cdots&0&H_{k}(D_{\bullet})_{[k]}&\cdots&H_{0}(D_{\bullet})_{[k]}\\
^{1}\E_{ver}:     &0&\cdots& &0                       &\cdots&0\\
H^{r}_{\frak{g}}(D_{r+s-1})_{[k]}\ar[r]&...\ar[r]^{\phi_k}&H^{r}_{\frak{g}}(D_{r+k})_{[k]}\ar@{--}^{\tau_k}[uurr]\ar[r]&0&0&\cdots&0}
\end{eqnarray*}

Since $H_j(D_{\bullet})$ is a sub-quotient of $D_j$, $H_j(D_{\bullet})_{[k]}=0,$ for all $j>k$. Also by (\ref{Ehor})
there is only one non-zero row in the horizontal spectral. On the other hand $\frak{g}$ is a regular sequence on the  $D_j$'s which in turn shows that the vertical spectral is just one line.

 Consequently $H_{j}(_k\mathcal{Z}^{+}_{\bullet}):=H_{j+r-1}(H^r_{\frak{g}}(D_{\bullet}))_{[k]}=H_{j-1}(D_{\bullet})_{[k]}=0$ whenever $j>k+1$. On the other hand, if $j<k$, $H_j(_k\mathcal{Z}^{+}_{\bullet}):=H_j(D_{\bullet})_{[k]}=H_{j+r}(H^{r}_{\frak{g}}(D_{\bullet})_{[k]})$. However the latter is zero since  $\End(H^{r}_{\frak{g}}(D_{r+j})\leq j$. This shows that $_k\mathcal{Z}^{+}_{\bullet}$
 is exact on the left and also on the right hand of $\tau_k$.

  It  remains to prove the exactness in the joint points $k$ and $k+1$.
	For this, notice that,  $(^{\infty}\E_{hor}^{-i,-j})_{[k]}=0$ for $j\geq 1$ hence $\mathcal{F}_{jk}/\mathcal{F}_{(j+1)k}=(^{\infty}\E_{hor}^{-k-j,-j})_{[k]}=0$ for $j\geq 1$. Consequently, $\mathcal{F}_{1k}=0$ in (\ref{F1}). Therefore the map $\tau_k$ defined in (\ref{tau}) is exactly the canonical map $H^{r}_{\frak{g}}(D_{r+k})_{[k]}\xrightarrow{Can.} \coker(\phi_k)_{[k]}$ as required.

\end{proof}

As already mentioned in the introduction, some sliding depth conditions are needed to prove the acyclicity of the $_k\mathcal{Z}^{+}_{\bullet}$ complexes.
\begin{Definition}\label{dsd} Let $(R,\fm)$ be a Noetherian local ring of dimension $d$ and $I=(f_1,...,f_r)$ an ideal of grade $g\geq 1$. Let $k$ and $t$ be two integers, we say that the
ideal $I$   satisfies $ \SD_k$ at level $t$ if
$\depth(H_i(\ff;R))\geq \min\{d-g,d-r+i+k\}$ for all $i\geq r-g-t$ (whenever
$t=r-g$ we simply say that $I$ satisfies $\SD_k$, also $\SD$ stands for
 $\SD_0$).

$I$ is strongly Cohen-Macaulay, SCM, if $H_i(\ff;R)$ is CM for all $i$. Clearly SCM is equivalent to $\SD_{r-g}$.

 Similarly, we say that $I$ satisfies the
sliding depth condition on cycles, $\SDC_k$, at level $t$, if
$\depth(Z_i(\ff,R))\geq \min\{d-r+i+k, d-g+2, d\}$ for all $i\geq r-g-t$. Again if $t=r-g$,  we simply say that $I$ satisfies $\SDC_k$ and we use $\SDC$ instead of $\SDC_0$.
\end{Definition}


Some of the basic properties and relations between conditions $\SD_k$  and $\SDC_k$ are explained in the following proposition.
\begin{Proposition}\label{psdsdc} Let $(R,\fm)$ be a CM local ring of dimension $d$ and $I=(f_1,...,f_r)$  be an ideal of grade $g\geq 1$. Let $k$ and $t$ be two integers. Then
\begin{enumerate}
\item[(1)]{The properties $\SDC_k$ and $\SD_k$ at
level $t$ localizes; they only depend  on $I$ and not the generating set, if $t=r-g$.}
\item[(2)]{$\SD_k$ implies $\SDC_{k+1}$. }
\item[(3)]{$\SD_k$ at level $t<r-g$ implies $\SDC_{a}$ at level $t$ for $a\leq g+t+2-d$. }
\item[(4)]{ $\SDC_{k+1}$ at level $t$ implies $\SD_{k}$ at level $t$ for any $t$, if $g\geq 2$. This implication is also the case if $g=1$ and  $k=0$. }
\item[(5)]{ $\SD_0$ at level $t\geq 1$ implies $\SDC_{0}$ at level $ t$.}
\end{enumerate}
\end{Proposition}
\begin{proof}$(1)$ is essentially proved in  \cite{V}. $(2)$ was already proved in  \cite[2.5]{Ha}. $(3)$ follows after analyzing the spectral sequences derived from the tensor product of  the truncated Koszul complex
\begin{center}
$0\longrightarrow Z_i(\ff) \longrightarrow K_i \longrightarrow K_{i-1} \longrightarrow \cdots \longrightarrow K_0\longrightarrow 0$.
\end{center}
and  the \v{C}ech complex, $\mathcal{C}^\bullet_\fm(R)$. To prove $(4)$, we consider the depth inequalities derived from the short exact sequences
$0\longrightarrow Z_{i+1}(\ff)\longrightarrow K_{i+1}\longrightarrow B_i(\ff) \longrightarrow 0,$ and
$0\longrightarrow B_i(\ff)\longrightarrow Z_{i}(\ff)\longrightarrow H_i(\ff) \longrightarrow 0.$
To show $(5)$ we use the fact that $\SDC_k$ holds for any $k$ at level $-1$ then a recursive induction
applying $\depth(Z_i)\geq \min\{\depth(H_i), d, \depth(Z_{i+1})-1\}$ proves the assertion.

\end{proof}

In the next Theorem we present sufficient conditions for the acyclicity of $_k\mathcal{Z}^{+}_{\bullet}$. The strategy taken here to prove the acyclicity is the same as the one applied in \cite{Ha} \textbf{(however the reader should notice that the complex $\mathcal{C}_{\bullet}$ defined in \cite{Ha} is a little bit different from  $_0\mathcal{Z}^{+}_{\bullet}$ here, in the former the tail is subsituted by a free complex but still   $\mathcal{C}_{\bullet}$ remains quasi-isomorphic to $_0\mathcal{Z}^{+}_{\bullet}$)}. Since the complex is finite  we avail ourselves of "lemme d'aciclicit\'e" of Peskine-Szpiro; so that we assume some sliding depth conditions and prove the acyclicity in  height $s-1$  wherein  $I=\fa$. By Lemma \ref{l1}, $_k\mathcal{Z}^{+}_{\bullet}$ is exact, if $I=\fa$. Then an induction will show the acyclicity globally.  

Although the proof here is  more involved than \cite[2.8]{Ha}, we prefer to omit it to go faster to newer theorems. The Cohen-Macaulay hypothesis in this theorem is needed  to show that if for an $R$-module $M$, $\depth(M)\geq d-t$ then for any prime $\fp$, $\depth(M_{\fp})\geq \Ht(\fp)-t$, \cite[Section 3.3]{V}.
\begin{Proposition} \label{t1}
Let $(R,\fm)$ be a CM local ring of dimension $d$ and $J=(\fa:I)$  be an $s$-residual intersection. Assume that $I=(f_1,...,f_r)$ and $\Ht(I)=g\geq 1$. Fix $0\leq k\leq \max\{s,s-g+2\}$. Then the complex $_k\mathcal{Z}^{+}_{\bullet}$ is acyclic, if any of the following hypotheses holds
\begin{enumerate}
\item[(1)] $1\leq s\leq 2$ and $s=k$, or

\item[(2)] $r+k\geq s+1$, $k\leq 2$ and $I$ satisfies  $\SDC_1$  at level $s-g-k$, or

\item[(3)] $r+k\leq s$ and $I$ satisfies  SD, or

\item[(4)] $r+k\geq s+1$, $k\geq 3$, $I$ satisfies  $\SDC_1$ at level $s-g-k$, and $\depth(Z_i(\ff))\geq d-s+k$, for $0\leq  i\leq k$.
\end{enumerate}
\end{Proposition}
Consequently, having a finite acyclic complex whose components have sufficient depth, one can estimate the depth of the zeroth homology.
\begin{Theorem} \label{t2}
Let $(R,\fm)$ be a  CM local ring of dimension $d$, $I=(f_1,...,f_r)$  be an ideal with $\Ht(I)=g \geq 1$. Let $s\geq g$ and fix $0\leq k\leq \min\{s,s-g+2\}$. Suppose that one the following hypotheses holds
\begin{itemize}
\item[(i)] $r+k\leq s$ and $I$ satisfies $\SD$ condition, or
\item[(ii)] $r+k\geq s+1$, $I$ satisfies $\SDC_1$ at level $s-g-k$ and $\depth(Z_i(\ff))\geq d-s+k$ for $0\leq i\leq k$, or
\item[(iii)] $k\leq s-r+2$ and $I$ satisfies $\SD$, or 
\item[(iv)] $\depth(H_i(\ff))\geq \min\{d-s+k-2,d-g\}$, for $0\leq i\leq k-1$ and $I$ satisfies $\SD$, or 
\item[(v)] $I$ is strongly Cohen-Macaulay.
\end{itemize}
Then for any $s$-residual intersection $J=(\fa:I)$, the complex $_k\mathcal{Z}^{+}_{\bullet}$ is acyclic. Furthermore,
$Sym_R^{k}(I/\fa)$ if $1\leq k\leq s-g+2$, or the disguised residual intersection if $k=0$, is  CM of codimension $s$.

\end{Theorem}


In the above theorem, the cases where $k=0,1$ and $s-g+1$ are the most important ones. For, $k=0$ is related to the disguised residual intersection, $k=1$ is connected to $R/\fa$ and $k=s-g+1$ is about the canonical module of $R/J$.

\begin{Remark} Any of the conditions $(iii)$-$(v)$ are stronger than $(i)$ and $(ii)$. Condition $(i)$ in Theorem \ref{t2} implies that $\depth(R/I)\geq d-s+k$ and  condition  $(v)$ 
implies that  $\depth(R/I)\geq d-g \geq d-s$. The other conditions $(ii),(iii)$ and $(iv)$ yield $\depth(R/I)\geq d-s+k-2$.
Therefore to have a better consequence of the theorem it may not be too much to ask $\depth(R/I)\geq d-s$.
\end{Remark}
The first two items in the next corollary settle conjectures $(3)$ and $(4)$ mentioned in the introduction.


\begin{Corollary}\label{cunmixed} Let $(R,\fm)$ be  a  CM local ring of dimension $d$, $I=(f_1,...,f_r)$ be an ideal with $\Ht(I)=g\geq 1$.
Suppose that  $\depth(R/I)\geq d-s$ and that $I$ satisfies any of the conditions in Theorem \ref{t2} for $k=1$, for instance if $I$ satisfies SD. Then for any algebraic $s$-residual intersection $J=(\fa:I)$, 
\begin{enumerate}
\item[(1)] $\depth(R/\fa)=d-s$ and  thus $\fa$ is minimally generated generated by $s$ elements.
\item[(2)] $J$ is  unmixed  of height $s$.
\item[(3)] $\Ass(R/\fa)\subseteq \Ass(R/I)\bigcup\Ass(R/J)$. Furthermore if $I$ is unmixed then the equality holds.
\item[(4)] If the residual is arithmetic then  $Sym_R^i(I/\fa)$ is a faithful $R/J$-module for all $i$.
\end{enumerate}
\end{Corollary}

\begin{proof}

$(1)$. Applying Theorem \ref{t2}  for $k=1$, we have $\depth(I/\fa)=d-s$. Therefore
the result follows from the exact sequence $0\ra I/\fa \ra R/\fa \ra R/I \ra 0$
and the fact that $\depth(R/I)\geq d-s$ by hypothesis.

$(2)$.  $J$ has codimension at least $s$, thus for any $\fp\in \Ass(R/J)$, $\Ht(\fp)\geq s$. On the other hand, $\Ass(R/J)\subseteq \Ass(R/\fa)$ and any prime ideal in the latter has codimension less than $s$ by $(1)$ (see \cite[1.2.13]{BH}), this yields the unmixedness of $J$.

$(3)$. Let $\fp \in \Ass(R/\fa)$, then $\Ht(\fp)\leq s$, by $(1)$. If $\Ht(\fp)< s$ or $\fp\not \supset J$, then $\fa_{\fp}=I_{\fp}$ and thus $\fp \in \Ass(R/I)$ otherwise $\Ht(\fp)=s$ and $\fp\supseteq J$ which clearly means $\fp \in \Ass(R/J)$. Verification of the last statement is straightforward.

$(4)$ To see this, we refer to the proof of  \cite[Theorem 2.11]{Ha} wherein it is shown that in the spectral sequence defined in (\ref{E2Hor}), $K:=(^{\infty}\E _{hor}^{0,0})_{[0]}\subseteq J\subseteq (^{2}\E _{hor}^{0,0})_{[0]}$. Moreover in the case where the residual is arithmetic this spectral sequence, locally in height $s$,  collapses in the second page, hence all of these inclusion are equality. We just notice that $( ^{2}\E _{hor}^{0,0})_{[0]}=H^{0}_{\frak{g}}(H_{0}(D_{\bullet}))_{[0]}=\bigcup_{i=0}^{\infty}(\bm{\gamma}\cdot  Sym_R^i(I):_R Sym_R^{i+1}(I))$. Thus once the equality in the above line happens, $J=\bm{\gamma}\cdot  Sym_R^i(I):_R Sym_R^{i+1}(I)$ for all $i$. Therefore $Sym_R^i(I/\fa)=Sym_R^{i+1}(I)/\bm{\gamma}\cdot  Sym_R^i(I)$ is a faithful $R/J$-module.

\end{proof}

The following examples show the accurancy of the conditions in Theorem \ref{t2} and Corollary \ref{cunmixed}(4).
\begin{Example}
Let $R=\mathbb{Q}[x_1,\cdots,x_7]$, $M=\left(\begin{array}{ccccclllll}
0&x_1&x_2&x_3\\x_4&x_5&x_6&x_7 \end{array}\right)$ and $I=I_2(M)$. Set $f_0=-x_4x_1, f_1=-x_2x_4,f_2=x_1x_6-x_2x_5,f_3=-x_3x_4,f_4=x_1x_7-x_3x_5$ and $f_5=x_2x_7-x_3x_6$.   \textit{Macaulay2} computations  show that $\depth(Z_1(I))=6, \depth(Z_2(I))=2$ and  $\depth(Z_3(I))=6$ hence $I$ satisfies SDC$_1$ at level $0$ but not at level $1$. However, $Z_1$ has enough depth, thence condition $(ii)$ in Theorem \ref{t2}
is fulfilled whenever $2\leq s\leq 4$. The non-trivial case will be $s=4$. Also $\depth(R/I)=4>d-s=3$. It then follows from Corollary \ref{cunmixed}  that for any $4$-residual intersection $J=\fa:I$, $\Ht(J)=4=\mu(\fa)$. It is easy to see that $I$ satisfies $G_{\infty}$; hence \cite[3.7bis]{EHU} implies that $\ell(I)\geq 5$.
On the other hand   the Pl\"{u}ker relations show that $f_2f_3-f_1f_4+f_0f_5=0$. Hence $I$ is not generated by analytically independent elements- that is $\ell(I)= 5$.
Therefore  the same proposition in \textit{loc. cit.} proves that there must exist  a $5$-residual intersection $J=\fa: I$ such that $\Ht(J)>5$. By the latter, Corollary \ref{cunmixed} guarantees that for that ideal $\fa$, the module $I/\fa$ cannot be CM and thus Theorem \ref{t2} does not work because SDC$_1$ at level $1$ is not satisfied.

Another point about this example is that $\depth(R/I^2)=2$, according to \textit{Macaulay2}; therefore $\Ext^5(R/I^2,R)\neq 0$. The vanishing of this $\Ext$ module would be  a  sufficient to prove that $I$ is $4$-residually $S_2$ in \cite[Theorem 4.1]{CEU} whereas we saw in the above corollary that $I$ is  $4$-residually unmixed.
\end{Example}
\begin{Example}
The arithmetic hypothesis in  Corollary \ref{cunmixed}(4) cannot be dropped. Let $R=\mathbb{Q}[x,y]$, $I=(x,y)$ and $\fa=(x^2,y^2)$. Then $J=(\fa , xy)$ and ${\rm Ann } (Sym^2_R(I/\fa))=(\fa I:I^2)=(x,y)$. The latter is not $J$, since the residual(linkage) is not arithmetic.


\end{Example}

In the case where the residual intersection is geometric we have
stronger corollaries.


 \begin{Corollary}\label{cgeometric} Let $(R,\fm)$ be a  CM local ring of dimension $d$, $I=(f_1,...,f_r)$ an ideal with $\Ht(I)=g\geq 1$.
  Suppose that $\depth(R/I)\geq d-s$ and that $I$ satisfies any of the conditions in Theorem \ref{t2} for some  $1\leq k\leq s-g+2$.
  Then for any geometric $s$-residual intersection $J=(\fa:I)$ and for that $k$, 
\begin{enumerate}
\item[(1)]$I^k/\fa I^{k-1}\simeq Sym_R^k(I/\fa) $ and it  is a faithful maximal  Cohen-Macaulay $R/J$-module.
\item[(2)]$\fa I^{k-1}=I^k\cap J$.
\item[(3)] $I^k+J$ is a CM ideal of height $s+1$ and $\Ht(I^k+J/J)=1$.
\end{enumerate}
\end{Corollary}
\begin{proof}

$(1)$. Consider the natural map $\varphi :\frac{Sym_R^k(I)}{\bm{\gamma} Sym_R^{k-1}(I)}\rightarrow \frac{I^k}{\fa I^{k-1}}$  that sends $T_i$ to $f_i$.
  $\varphi$ is onto. Let $K=\Ker(\varphi)$. We show that $\Ass(K)=\emptyset$, hence $K=0$.  $\Ass(K)\subseteq \Ass(\frac{Sym_R^k(I)}{\bm{\gamma} Sym_R^{k-1}(I)})$. By Theorem \ref{t2},
$\depth(\frac{Sym_R^k(I)}{\bm{\gamma} Sym_R^{k-1}(I)})=d-s$, hence for any associated prime $\fp$ in the latter $\Ht(\fp)\leq s$. However, since $J$ is a
geometric residual, for any $\fp$ of codimension $\leq s$, either $I_{\fp}=\fa_{\fp}$ or $I_{\fp}=(1)$. In the former
 case $\frac{Sym_R^k(I_{\fp})}{\bm{\gamma}_{\fp} Sym_R^{k-1}(I_{\fp})}=0$, for $k\geq 1$, which implies that $\fp\not \in \Ass(K)$.
 In the latter case, $Sym_R^k(I_{\fp})=R_{\fp}$, hence $\varphi_{\fp}$ is an isomorphism, which means $K_{\fp}=0$ in particular $\fp\not \in \Ass(K)$. Therefore $\Ass(K)=\emptyset$, hence $K=0$.  The CM ness of $Sym_R^k(I/\fa)$ was already
 shown in Theorem \ref{t2}. To see that $I^k/\fa I^{k-1}$ is faithful one may use Corollary \ref{cunmixed} and part $(1)$, or,  just notice that $J\subseteq \fa I^{k-1}:I^k $ and that
 the equality holds locally at all associated primes of $J$ which are of height $s$.

 $(2)$. Consider the canonical map $\psi:\frac{I^k}{\fa I^{k-1}}\rightarrow \frac{I^k+J}{J}$. By $(2)$, $I^k/\fa I^{k-1}$ is
 CM of dimension $d-s$ which  in conjunction with the fact that the residual is geometric shows that $\psi$ is an isomorphism for all $0\leq k\leq s-g+2$.
 Notice that $\frac{I^k+J}{J} \simeq \frac{I^k}{J\cap I^{k}} $. Thus the canonical map imply $J\cap I^{k}=\fa I^{k-1}$.

 $(3)$. By $(1)$ and $(2)$,  $(I^k+J)/J$ is CM of dimension $d-s$. Also by  Theorem \ref{t2}, $R/J$ is CM of dimension $d-s$. Thus the exact sequence
$0\ra (I^k+J)/J\ra R/J\ra R/(I^k+J)\ra 0$
implies $\depth(R/I^k+J)\geq d-s-1$. On the other hand, since the residual is geometric, $\Ht(I^k+J)\geq s+1$ which in turn completes the proof. To see that $\Ht(I^k+J/J)=1$, it is just enough to notice that $R/J$ is CM.

\end{proof}

We now proceed to determine the structure of the canonical module of
$R/J$. This achievement is via  a study of the type of an appropriate
candidate
 for the canonical module. Recall that in a Notherian local ring $(R,\fm)$ the type of a finitely generated module $M$
 is the dimension of  the $R/\fm$-vector space   $\Ext^{\depth(M)}_R(R/\fm,M)$ and it is denoted by $r_R(M)$ or just $r(M)$.

\begin{Theorem}\label{ttype} Let $(R,\fm)$ be a  CM local ring of dimension $d$ and  $I=(f_1,...,f_r)$  be an ideal with $\Ht(I)=g\geq 2$.
 Let $J=(\fa:I)$ be an $s$-residual intersection  of $I$ and  $1\leq k\leq s-g+1$. Suppose that $I$ is SCM.
Then $$r_R(Sym_R^{k}(I/\fa))\leq \binom{r+s-g-k}{r-1}r_R(R).$$

\end{Theorem}
\begin{proof}

First we deal with the case where $g=2$. Consider the complex
$_k\mathcal{Z}^{+}_{\bullet}$
\begin{equation}\label{TZ'+}
\xymatrix{0\ar[r]&H_{\frak{g}}^{r}(D_{r+s-1})_{[k]}\ar[r]&\cdots\ar^{\phi_k}[r]&H_{\frak{g}}^{r}(D_{r+k})_{[k]}\ar[r]^{\tau_k}&
(D_{k})_{[k]}\ar[r]&\cdots\ar[r]&(D_{0})_{[k]}\ar[r]&0}
\end{equation}
 We study the depth of the components of this complex.  $Z_{r-1}=R$  hence $\depth(_k\mathcal{Z}^{+}_{r+s-1})=d$. For $k\leq j\leq s-2$, the first cycle which appears in $D_{r+j}$ is
 $Z_{r-s+j}$ and it has depth equal to $d=d-g+2$, by SCM condition, which is at least $d-s+(j+1)+1$.
  As for the right part of the
 complex we have $\depth(Z_i)=d-g+2\geq d-s+k+1 $  for all $0\leq i \leq k$, since by
 assumption $k\leq s-g+1$.

  Summing up we have \begin{equation}\label{depthz'}\depth(_k\mathcal{Z}^{+}_{i})\geq
 d-s+i+1 \text{~~for all~~} i=0,\cdots,s-1.\end{equation}

Now, let $F_{\bullet}\ra R/\fm$ be a free resolution of $R/\fm$. We
put the double complex
$\Hom_R(F_{\bullet},~_k\mathcal{Z}^{+}_{\bullet})$ in the third
quadrant with $\Hom_R(F_0,~_k\mathcal{Z}^{+}_{0})$ in the center.
The two arisen spectral sequences are:
\begin{equation*}
^\infty \E^{-i,-j}_{hor}=\left \{
\begin{array}{cc}
\Ext_R^{j}(R/\fm, Sym_R^{k}(I/\fa)),& i=0 \\
 0& i\neq 0,
\end{array}
\right.
\end{equation*}
and $^1 \E^{-i,-j}_{ver}=\Ext_R^{j}(R/\fm,~
_k\mathcal{Z}^{+}_{i})$. The latter vanishes for  all $(i,j)$ where
$j-i\leq d-s$ and $i\neq s$,  according to (\ref{depthz'}). The
convergence of both complexes to the homology module of the total
complex implies that $\Ext_R^{d-s}(R/\fm, Sym_R^{k}(I/\fa))\simeq~~ ^\infty
\E^{-s,-d}_{ver}$. The latter is a submodule of $\Ext_R^{d}(R/\fm,~
_k\mathcal{Z}^{+}_{s})$. Therefore $r_R(Sym_R^{k}(I/\fa))\leq \Dim_{R/\fm}(
\Ext_R^{d}(R/\fm,~ _k\mathcal{Z}^{+}_{s}))$.

Following  from the construction of $\mathcal{D}_{\bullet}$,
$D_{r+s-1}=S(-r+1-s)$. Hence
$$_k\mathcal{Z}^{+}_{s}=H^r_{\frak{g}}(S)_{[-r+1-s+k]}.$$
To calculate the dimension of the above Ext module we just need to
know how many copies of $R$ appear in the inverse polynomial
structure of the above local cohomology module. The answer is simply
the positive solutions of the numerical equation
$\A_1+\cdots+\A_r=r-1+s-k$ which is $\binom{r-k+s-2}{r-1}$.
Therefore $r_R(Sym_R^{k}(I/\fa))\leq \binom{r+s-2-k}{r-1}r(R)$ in the case where $g=2$.

Now, let $\Ht(I)=g\geq 2$. We choose a regular sequence $\mathbf{a}$ of length $g-2$ inside $\fa$ which is a part of a minimal generating set of $\fa$.
$I/\aa$ is still SCM by \cite[Corollary 1.5]{H}, moreover $J/\aa=\fa/\aa:I/\aa$. Recall that $\mu(\fa)=s$ by Corollary  \ref{cunmixed}, therefore $J/\aa$ is an $(s-g+2)$-residual intersection of $I/\aa$ which is of height $2$. Hence we return to the  case $g=2$, so that
 $$r_R(Sym_R^{k}(I/\fa))=r_{R/\aa}(Sym_{R/\aa}^{k}(I/\fa))\leq \binom{r+(s-g+2)-2-k}{r-1}r_{R/\aa}(R/\aa)=\binom{r+s-g-k}{r-1}r_R(R).$$
(see \cite[Exercise
1.2.26]{BH}).

\end{proof}

\begin{Remark} Clearly the above proof works  as well for $k=0$, hence if $I$ is SCM of $\Ht(I)\geq 2$ and $K$ is a disguised $s$-residual intersection of $I$ then  $$r_R(R/K)\leq \binom{r+s-g}{r-1}r_R(R).$$
\end{Remark}

Everything is now ready to explain the structure of the canonical module.
\begin{Theorem}\label{ccanonical}Suppose that $(R,\fm)$ is a  Gorenstein local ring of dimension $d$. Let  $I$ be a SCM ideal with $\Ht(I)=g$ and  $J=(\fa:I)$ be an arithmetic $s$-residual intersection  of $I$. Then $\omega_{R/J}\simeq Sym_R^{s-g+1}(I/\fa)$. It is isomorphic to $I^{s-g+1}+J/J$ in the geometric case; Furthermore  $R/J$ is generically Gorenstein and,  $R/J$ is Gorenstein if and only if $I/\fa$ is principal.
\end{Theorem}
\begin{proof}
Without loss of  generality, we may suppose that $g \geq 2$, otherwise we may add a new variable $x$ to $R$ and consider ideals $\fa+(x)$ and $I+(x)$ in the local ring $R[x]_{\fm+(x)}$ wherein we have $J+(x)=(\fa+(x)):(I+(x))$.

$R/J$ is CM by Theorem \ref{t2}. Respectively, according to Theorem \ref{t2}, Corollary \ref{cunmixed} and Theorem \ref{ttype},  $Sym_R^{s-g+1}(I/\fa)$ is a maximal CM, faithful $R/J$-module and of type $1$. Therefore it is the canonical module of $R/J$ by \cite[3.3.13]{BH}. In the geometric case $Sym_R^{s-g+1}(I/\fa)\simeq I^{s-g+1}+J/J$ by Corollary \ref{cgeometric}.

To see that $R/J$ is generically Gorenstein, notice that for any prime ideal $\fp\supseteq J$ of height $s$. $r((R/J)_{\fp})=\mu((\omega_{R/J})_{\fp})=\mu(Sym_{R_{\fp}}^{s-g+1}((I/\fa)_{\fp}))$ which is $1$ since the residual is arithmetic by the assumption.  Finally let $t=\mu(I/\fa)$ then $r(R/J)=\mu(\omega_{R/J})=\binom{s-g+t}{t-1}$ which is one if and only if $t=1$.
\end{proof}


\section{Hilbert Function,  Castelnuovo-Mumford regularity}\label{App}

Athough the family $\{_k\mathcal{Z}^{+}_{\bullet}\}$ does not consist of free complexes, however it
can approximate $R/J$ and $Sym_R^k(I/\fa)$ very closely. Several numerical invariants or functions such as Castelnuovo-Mumford regularity,
Projective dimension and Hilbert function can be estimated via this family. In \cite{CEU}, Chardin, Eisenbud and Ulrich restated an old question
(\cite{S}) asking for which open sets of ideals $\fa$ the Hilbert function of $R/\fa$ depends only on the degrees of the generators of $\fa$. In \cite{CEU}, the authors consider the following two conditions
\begin{itemize}
\item[(A1)]If the Hilbert function of $R/\fa$ is constant on the open set of ideals $\fa$ generated by $s$
forms of the given degrees such that codim$(\fa : I)\geq s$; and
\item[(A2)]if the Hilbert function of $R/(\fa : I)$ is constant on this set.
\end{itemize}
It is shown in \cite[Theorem 2.1]{CEU} that ideals with some slight depth conditions in conjunction with  $G_{s-1}$ or $G_s$  satisfy these two conditions. In this direction we have the following result which in addition to show the validity of the above conditions $(A1)$ and $(A2)$ under some sliding depth condition, provides a method to compute the desired Hilbert functions.

\begin{Proposition}\label{phil} Let $R$ be a CM graded ring over an Artinian local ring $R_0$, $I$ and $\fa$ be two homogenous ideals of $R$. Let $J=(\fa:I)$ be an $s$-residual intersection. Let $0\leq k\leq s-g+2$ and suppose that $I$ satisfies any  of the following hypotheses:
\begin{enumerate}
\item[(1)] $1\leq s\leq 2$ and $s=k$, or
\item[(2)] $r+k\geq s+1$, $k\leq 2$ and $I$ satisfies  $\SDC_1$  at level $s-g-k$, or
\item[(3)] $r+k\leq s$ and $I$ satisfies  SD, or
\item[(4)] $r+k\geq s+1$, $k\geq 3$, $I$ satisfies  $\SDC_1$ at level $s-g-k$, and $\depth(Z_i(\ff))\geq d-s+k$, for $0\leq  i\leq k$.
\end{enumerate}
  Then the Hilbert function of  the disguised $s$-residual intersection, if $k=0$,  and that of $Sym_R^k(I/\fa)$, if $1\leq k\leq s-g+2$, depends only on the \textbf{degrees} of the generators of $\fa$ and the Koszul homologies of $I$.
In particular if $I$ satisfies  any of the above conditions for $k=1$ then the Hilbert function of $R/\fa$ satisfies $(A1)$.
\end{Proposition}
\begin{proof} By Proposition \ref{t1}, the standing assumptions on  $I$  imply the acyclicity of $_k\mathcal{Z}^{+}_{\bullet}$. Hence the Hilbert function of  $H_0(_k\mathcal{Z}^{+}_{\bullet})$ is derived from the Hilbert functions of the components of   $_k\mathcal{Z}^{+}_{\bullet}$ which, according to (\ref{Di}) and (\ref{z+}), are just some direct sums of the  Koszul's cycles of $I$ shifted by the twists appearing in the koszul complex $K_{\bullet}(\bm{\gamma},S)$. Since the Hilbert functions of Koszul's cycles are inductively calculated in terms of those of  the Koszul's homologies, the Hilbert function of $H_0(_k\mathcal{Z}^{+}_{\bullet})$ depends on the  Koszul's homologies of $I$ and just the degrees of the generators of $\fa$. In the case where $k=1$ we know the Hilbert function of $I/\fa=Sym_R^1(I/\fa)$ so that we know the Hilbert function of $R/\fa$ by $0\ra I/\fa\ra R/\fa\ra R/I\ra0$. 
\end{proof}


The next important numerical invariant associated to an algebraic or geometric object is the Castelnuovo-Mumford regularity. Here we present an upper bound for the regularity of the disguised $s$-residual intersection $K$ and the regularity of  $Sym_R^k(I/\fa)$ (if $1\leq k\leq s-g+1$). 

Assume that
$ R=\bigoplus _{n \geq 0} R_{n}$ is a
 positively graded *local Noetherian ring of dimension $d$ over a Noetherian local ring $
(R_{0}, \fm_{0})$ set $\fm=\fm_{0}+R_+$. Suppose that $I$ and
$\fa$ are homogeneous ideals of $R$ generated by homogeneous
elements $f_1,\cdots,f_r$ and $a_1,\cdots,a_s$, respectively.
 Let $\deg f_t=i_t$ for all $1 \leq t \leq r$ with $i_1
\geq \cdots \geq i_r$ and $\deg a_{t}=l_{t}$ for $1\leq t \leq
s$. For a graded ideal $\frak{b}$, the sum of the degrees of a
minimal generating set of $\frak{b}$ is denoted by
$\sigma(\frak{b})$.

\begin{Proposition} \label{Pregularity}
With the same notations just introduced above, let $(R,\fm)$ be a  CM $^*$local ring, $I$ an ideal with $\Ht(I)=g \geq 2$ and  $s\geq g$. Suppose that $J=(\fa:I)$ is  an $s$-residual intersection of $I$ and  $K$ is the disguised $s$-residual intersection of $I$ w.r.t $\fa$. Suppose that any of the  following conditions hold for  some $0\leq k\leq s-g+1$,
\begin{itemize}
\item[(i)] $r+k\leq s$ and $I$ satisfies $\SD_1$ condition;
\item[(ii)] $r+k\geq s+1$, $I$ satisfies $\SDC_2$ at level $s-g-k$ and $\depth(Z_i(\ff))\geq d-s+k+1$ for $0\leq i\leq k$.
\end{itemize}
Then if $k\geq 1$, \begin{center}
 $\Reg (Sym_R^k(I/\fa)) \leq \Reg( R) + \Dim(R_0)+\sigma( \fa)-(s-g+1-k)\beg(I/\fa)-s $.
\end{center}
 And if $k=0$,  $\Reg (R/K) \leq \Reg( R) + \Dim(R_0)+\sigma( \fa)-(s-g+1)\beg(I/\fa)-s $.
\end{Proposition}
\begin{proof} The proof of this fact is essentially the same as that of \cite[Theorem 3.6]{Ha} wherein the case of $k=0$  is verified.  The crucial point to repeat that proof is that we have to change the structure of the families $\{_k\mathcal{Z}^{+}_{\bullet}\}$ from the beginning   by substituting the tails with a free complex. Another way to prove this result in the case where $I$ is SCM, may be by applying the "$g=2$" trick in Theorem \ref{ttype}.
\end{proof}

The inequality in the above proposition becomes more interesting in particular cases, for example if $R$ is a polynomial ring over a field and $k=1$ this inequality reads as $$\Reg_R (I/\fa)+ (s-g)\beg(I/\fa) \leq \sigma( \fa)-s .$$ Hence we have a numerical obstruction for $J= (\fa:I)$ to be an $s$-residual intersection.

At the other end of the spectrum we have $k=s-g+1$, this case  corresponds to the canonical module of $R/J$. 
 
 The next proposition improves \cite[Proposition 3.5]{Ha} by removing the $G_s$ condition.

\begin{Proposition}\label{Pcanograded} Suppose that $(R,\fm)$ is a positively graded Gorenstein *local
 ring, over a Noetherian local ring $
(R_{0}, \fm_{0})$, with canonical module $\omega_R=R(b)$ for some integer $b$. Let  $I$ be a homogeneous SCM ideal with $\Ht(I)=g$, $\fa\subset I$  homogeneous and
 $J=(\fa:I)$ be an arithmetic $s$-residual intersection  of $I$.  Then $\omega_{R/J}=Sym_R^{s-g+1}(I/\fa)(b+\sigma(\fa))$.
 
 In particular, if $\Dim(R_0)=0$ then $\Reg (R/J) = \Reg( R) + \sigma( \fa)-\beg(Sym_R^{s-g+1}(I/\fa))-s. $
\end{Proposition}
\begin{proof} We may and do assume that $g\geq 2$. First suppose that $g=2$. We keep the notations defined at the beginning of the section $2$ to define  $_k\mathcal{Z}^+_{\bullet}$ complex. Considering $R$ as a subalgebra of $S=R[T_1,\cdots,T_r]$, we write the
degrees of an element $x$ of $R$ as the $2$-tuple $(\Deg x,0)$with the second entry
zero. So that $\deg f_t=(i_t,0)$ for all $1 \leq t \leq r$,
$\deg a_t=(l_t,0)$ for all $1 \leq t \leq s$, $\deg
T_t=(i_t,1)$ for all $1 \leq t \leq r$, and thus $\deg
\bm{\gamma}_t=(l_t,1)$ for all $1 \leq t \leq s$. With these
notations the $\mathcal{Z}$-complex has the following shape
\begin{center}
$\mathcal{Z}_\bullet: 0\rightarrow
Z_{r-1}\bigotimes_{R}S(0,-r+1) \rightarrow \cdots \rightarrow
Z_1\bigotimes_{R}S(0,-1) \rightarrow
Z_0\bigotimes_{R}S\rightarrow 0.$
\end{center}
Consequently,
 \begin{center}
$\mathcal{Z}_{r-1}=R(-\sum_{t=1}^{r}
i_t,0)\bigotimes_{R}S(0,-r+1)$,
\end{center}
and, taking into account that $a_1,\cdots,a_s$ is a minimal generating
set of $\fa$ by Corollary \ref{cunmixed},
\begin{center}
$D_{r+s-1}=S(-\sum_{t=1}^{r}
i_t-\sigma(\fa),-s-r+1))$.
\end{center}
It then follows that $_{(s-1)}\mathcal{Z}^+_{s}=H^r_{\frak{g}}(D_{r+s-1})_{[\ast,s-1]}$ is isomorphic to
\begin{equation}\label{e(r,s)}
H^r_{\frak{g}}(S)(-\sum_{t=1}^{r}i_t-\sigma(\fa),0)_{[\ast,-r]}\simeq (RT_1^{-1}\cdots T_r^{-1}(-\sum_{t=1}^{r}i_t))(-\sigma(\fa))\simeq R(-\sigma(\fa)).
\end{equation}
Notice that $\Deg(T_j^{-1})=(-i_j,-1)$.

We next turn to the proof of Theorem \ref{ttype} and consider the spectral sequences therein for $k=s-g+1=s-1$. We have $^1 \E^{-s,-d}_{ver}=\Ext_R^{d}(R/\fm,~
_{(s-1)}\mathcal{Z}^{+}_{s})$ which is isomorphic to $\Ext^d_R(R/\fm,R(b))(-b-\sigma(\fa))$ by (\ref{e(r,s)}).
 The latter is in turn  homogeneously isomorphic to $R/\fm(-b-\sigma(\fa))$, since $R$ is Gorenstein. The inclusion  $\Ext_R^{d-s}(R/\fm, Sym_R^{s-g+1}(I/\fa))\simeq ~~^\infty
\E^{0,-(d-s)}_{hor}\hookrightarrow~~ ^1 \E^{-s,-d}_{ver}=\Ext_R^{d}(R/\fm,~
_{s-1}\mathcal{Z}^{+}_{s})$ is indeed an isomorphism, since the latter is a vector space of dimension one.
In the graded case all of the homomorphisms in the proof of  Theorem \ref{ttype} are homogeneous,
 therefore the above isomorphism
 implies that $\Ext_R^{d-s}(R/\fm, Sym_R^{s-g+1}(I/\fa))\simeq R/\fm(-b-\sigma(\fa)).$
 Since we already know that in the local case
 $\omega_{R/J}=Sym_R^{s-2+1}(I/\fa)$ by Corollary \ref{ccanonical}, hence in the graded case
   we have  $\omega_{R/J}=Sym_R^{s-2+1}(I/\fa)(b+\sigma(\fa))$.

Now suppose that $g\geq 2$ and let $\aa=a_1,\cdots,a_{g-2}$ be a regular sequence in
$\fa$ which is a part of its minimal generating set. Then $\omega_{R/\aa}=\frac{R}{\aa}(b+\sigma(\aa))$
 and moreover we go back to the case where $g=2$. We then have
 $$\omega_{R/J}=Sym_{R/\aa}^{(s-g+2)-2+1}\left(\frac{I/\aa}{\fa/\aa}\right)((b+\sigma(\aa))+\sigma(\fa/\aa))=Sym_R^{s-g+1}(I/\fa)(b+\sigma(\fa)).$$

\end{proof}


Finally, one may ask whether under the $G_s$ condition our methods can contribute more information about the structure of residual intersections
 than what is known so far.
The next result shows that in the presence of the $G_s$ condition the complex $_0\mathcal{Z}^{+}_{\bullet}$ approximates $R/J$ for any (algebraic)
 $s$-residual intersection $J$. In particular   the above facts about the regularity and the Hilbert function hold for $R/J$.

 \begin{Proposition}\label{PGs}Let  $(R,\fm)$ be a  Gorenstein local ring and suppose that $I$ is  SCM    evenly linked to an ideal which satisfies $G_s$.
 Then  the disguised $s$-residual intersections of $I$  with respect to an ideal $\fa$, $K$, is the same as the algebraic residual intersection $J=\fa:I$.

 \end{Proposition}
 \begin{proof}   It is shown in \cite[Theorem 5.3]{HU}
that under the above assumptions on $I$ there exists a deformation $(R',I')$ of $(R,I)$ such that $I'$ is SCM and $G_s$.
Let $I=(f_1,\cdots,f_r)$,  $\fa=(a_1,\cdots,a_s)$ and suppose that $a_i=\sum_{j=1}^rc_{ij}f_j$. Let $s_{ij}$ be a lifting of $c_{ij}$ to $R'$ and
$f'_{i}$ be a lifting of $f_i$ to $R'$.
 Set $X=(x_{ij})$ to be an $s\times r$ matrix of invariants, $\bm{\alpha}:=(\A_1,\cdots,\A_s)=X\cdot(f'_1,\cdots,f'_r)^t$.
Let $\tilde{R}=R'[x_{ij}]_{(\fm+ (x_{ij}-s_{ij}))}$ and  $\tilde{J}:=(\A_1,\cdots,\A_s):_{\tilde{R}}I'\tilde{R}$.
 By  \cite[Lemma 3.2]{HU}  $\tilde{J}$ is a geometric $s$-residual intersection of $I'\td{R}$.
We construct the complex $_0\mathcal{Z}^{+}_{\bullet}((\A_1,\cdots,\A_s);(f'_1,\cdots,f'_s))$ in $\td{R}$
and set $\td{R}/\td{K}=H_0(_0\mathcal{Z}^{+}_{\bullet}((\A_1,\cdots,\A_s);(f'_1,\cdots,f'_s)))$. Since $\tilde{J}$  is a geometric residual, $\tilde{J}=\tilde{K}$. 

 Assume that  $\pi$ is  the projection from $\td{R}$ to $R$ which sends $x_{ij}$ to $c_{ij}$.
 By  \cite[Proposition 2.13]{Ha}, $R/\pi(\td{K})=H_0(_0\mathcal{Z}^{+}_{\bullet})$
and by \cite[Theorem 4.7]{HU}, $\pi(\td{J})=J$. Therefore $R/K=R/\pi(\td{K})=R/\pi(\td{J})=R/J$.
 \end{proof}


\section{Arithmetic Residual Intersections}\label{Arith}
In this section we scrutinize the structure of the $_0\mathcal{Z}^{+}_{\bullet}$ complex in the case where $I/\fa$ is cyclic. We find that the homologies of $_0\mathcal{Z}^{+}_{\bullet}$ determines the uniform annihilator of non-zero Koszul homologies. The importance of this fact is that although the Koszul complex is an old object, not much is known about the annihilator of higher homologies. Besides the rigidity and differential graded algebra structure of the homologies, there is an interesting paper of Corso, Huneke, Katz and Vasconcelos \cite{CHKV} wherein the authors make efforts to find out "whether the annihilators of non-zero koszul homology modules of an unmixed ideal is contained in the integral closure of that ideal". However, it seems that the situation of mixed ideals is more involved, for example  for $\fa=(x^2-xy,y^2-xy,z^2-zw,w^2-zw)$, (taken from \textit{loc. cit.}), $\Ann(H_1)=\bar{\fa}$,  the integral closure, and $\Ann(H_2)=\sqrt\fa$. This example shows that the intersection of the annihilators has a better structure than the union. This is an interesting example for us since $(\fa:\bar{\fa})$ is a $4$-residual intersection with $\bar{\fa}/\fa$ cyclic  and $\bar{\fa}$  satisfies $\SD$.

Since a part of our results works well with arithmetic $s$-residual intersections, it  is worth mentioning some general properties of these ideals.
\begin{Proposition}\label{Parith} Let $R$ be a Noetherian ring, $\fa= (a_1,\cdots,a_s)\subset I$ be  ideals of $R$ and $J=\fa:I$ be an $s$-residual intersection.  Then
\begin{itemize}
 \item[(i)] $J$ is an arithmetic $s$-residual intersection if and only if $\Ht(0:\bigwedge^{2}(I/\fa))\geq s+1$.
 \item[(ii)] If $2$ is unit in $R$ then $I+J\subseteq (0:\bigwedge^{2}(I/\fa))$.
 \item[(iii)]The followings statements are equivalent
 \begin{enumerate}
  \item For all $i\leq s$, an $i$-residual intersection of $I$ exists.
  \item For any prime ideal $\fp\supset I$ such that $\Ht(\fp)\leq s-1$, $\mu(I_{\fp})\leq \Ht(\fp)+1$.
 \item There exist $a_1,\cdots,a_s\in I$ such that for $i\leq s-1$, $J_i=(a_1,\cdots,a_i):I$ is an arithmetic $i$-residual intersection and  $J=(a_1,\cdots,a_s):I$ is an (algebraic) $s$-residual intersection.
 \end{enumerate}
 \item[(iv)] $J$ is an arithmetic $s$-residual intersection if and only if there exists an element $b\in I$ such that $\Ht((\fa,b):I)\geq s+1$.
\end{itemize}
\end{Proposition}
\begin{proof} (i). $J=\fa:I$ is an arithmetic $s$-residual if and only if $\Ht({\rm Fitt}_j(I/\fa))\geq s+j$ for $j=0,1$ \cite[Proposition 20.6]{Eis}. Moreover ${\rm Fitt}_j(I/\fa)$ and the  annihilator of $\bigwedge^{j+1}(I/\fa)$ have the same radical \cite[Exercise 20.10]{Eis} which yields the assertion.
(ii) is  straightforward.
(iii) goes along the same line as the proof of \cite[Lemma 1.4]{U}. (iv) follows from \cite[Lemma 1.3]{U}.
\end{proof}

The concept of $s$-parsimonious ideals first appeared in \cite{CEU}. It is shown  that  weakly $s$-residually $S_2$ ideals which satisfy $G_s$ are  $s$-parsimonious \cite[3.1]{CEU}. However if $I$ is an ideal which satisfies the condition of Corollary \ref{cunmixed} then it follows from  \cite[3.3 (a)]{CEU} that $I$ is $s$-parsimonious. So that we have
\begin{Remark}Let $(R,\fm)$ be  a  CM local ring of dimension $d$, $I$ an ideal such that  $\depth(R/I)\geq d-s$ and that $I$ satisfies any of the conditions  in Theorem \ref{t2} for $k=1$. Then for any arithmetic $s$-residual intersection $J=(\fa:I)$, there exists an element $b\in I$ such that $J=(\fa:b)$.
\end{Remark}

This remark is another evidence to study the   cases where $I/\fa$ is cyclic.

To exploit the structure of $H_{\bullet}(_0\mathcal{Z}^{+}_{\bullet}(\fa;\fa,b))$, the following lemma is crucial.

\begin{Lemma}\label{lr2} Let $\fa=(a_1,\cdots,a_s)$, $I=(\ff)=(b,a_1,\cdots,a_s)$ and consider the complex $\mathcal{D}_{\bullet}$
defined in (\ref{Di}). Let $Z'_{\bullet}$, $B'_{\bullet}$ and $H'_{\bullet}$  be the cycles, boundaries and homologies of the Koszul complex
$K_{\bullet}(a_1,\cdots,a_s)$. Put $\widetilde{B_{i}}=(B'_{i}:_{Z'_i}b)$ and $S=R[T_0,T_1,\cdots,T_s]$ with standard grading. Then
\begin{itemize}
\item[(i)] $$H_i(\mathcal{D}_{\bullet})\cong \frac{(Z'_{i}\y_R R[T_0])(-i)}{(\widetilde{B_{i}}\y_R T_0R[T_0])](-i)},$$
in particular
$$(H_{i}(\mathcal{D}_{\bullet})_{T_0})_{[0]}\cong bH_i(a_1,\cdots,a_s)~~\text{for all~~} i.$$
\item[(ii)] If $R$ is a graded ring and $b, a_1,\cdots,a_s$ are homogeneous then we have a homogeneous isomorphism
$$(H_{i}(\mathcal{D}_{\bullet})_{T_0})_{[0]}\cong bH_i(a_1,\cdots,a_s)((i+1)\deg(b))~~\text{for all~~} i.$$
\end{itemize}
\end{Lemma}
\begin{proof} To compute the homology modules of $\mathcal{D}_{\bullet}$, we consider the two spectral sequences arising from its double complex structure $(\ref{Di})$. We put this double complex in the second quadrant and locate $K_0(\bm{\gamma})\y_S\mathcal{Z}_0(\ff)$ in the corner.

 Set $S=R[T_0,T_1,\cdots,T_s]$ where $T_0$ is corresponds to $b$ and $T_i$ to $a_i$. Then  $(\gamma_1\cdots\gamma_s)=(T_1\cdots T_s)$   is a regular sequence on $\mathcal{Z}_i=Z_i[T_0,T_1,...,T_s](-i)$
 with $Z_{i}=Z_{i}(b,a_1,\cdots,a_s)$.
 We then have $^1\E^{-i,j}_{ver}=H_j(K_{\bullet}(\ag)\y \mathcal{Z}_i)=0$ for all $j\geq 1$ and all $i$. Also $ (S/(T_1,T_2,...,T_s))\y_S\mathcal{Z}_i=R[T_0]\y_R Z_{i}$ in which we consider $R[T_0]$ as an $S$-module by trivial  multiplication $T_iR[T_0]=0$, $1\leq i\leq s$.
  It then  follows that $^1E_{ver}^{0,\bullet}$ is the complex
\begin{equation}\label{Ever}
\begin{array}{cccccc}
0\ra Z_{s+1}\y_R R[T_0](-s-1)\xrightarrow{\partial_{s+1}} \cdots\ra Z_{1}\y_R R[T_0](-1) \xrightarrow{\partial_1} Z_{0}\y_R R[T_0]\ra 0
\end{array}
 \end{equation}
wherein the differentials are induced by  those in $\mathcal{Z}_{\bullet}(\ff)$. As well,
\begin{equation}\label{Dconverge}
H_{i}(^1E_{ver}^{0,\bullet})=H_{i}(\mathcal{Z}_{\bullet}\y_S (S/(T_1,T_2,...,T_s)))(-i)
\end{equation}

 Considering the convergence of the spectral sequences to the
homology module of the total complex $\mathcal{D}_{\bullet}$, we have $H_i(\mathcal{D}_{\bullet})=H_{i}(^1E_{ver}^{0,\bullet})= \Ker(\partial_i)/\Image(\partial_{i+1})$. We compute this homology in two steps.

{\bf Step 1 $\Ker(\partial_i)$:} Let $z_i=( \sum_{j} r_{j}e_{j_1}\wedge e_{j_2}\wedge ...\wedge e_{j_i})\omega (T_0)\in Z_i[T_0](-i)$.
We study two cases: First, if $j_1\neq 0$ for every $j$ which appears in the above presentation of $z_i$, in this case $z_i\in Z'_{i}[T_0]$
and moreover
$$\partial_{i}(z_i)= \sum_{j} r_{j}( \sum_{l=1}^{i}e_{j_1}\wedge ...\wedge e_{j_{l-1}}\wedge e_{j_{l+1}}\wedge...\wedge e_{j_i}T_l\omega (T_0))=0$$
 where the last vanishing is due to the fact that $T_iR[T_0]=0$ for all $1\leq i\leq s$. It follows that $z_i\in \Ker(\partial_i)$.

Second, we show that if $j_1=0$ for some $j$ that appears in $z_i$ then $z_i\not\in \Ker(\partial_i)$.  Writing $z_i=z'_i+z^{''}_i\wedge e_0$, where $z^{'}_i$ and  $z_i^{''}$ do not involve $e_0$, one has
$$\partial_i(z_i)=\partial_i(z'_i)+\partial_{i-1}(z^{''}_i)\wedge e_0+(-1)^{i+1}T_0z^{''}_i$$ by the  DG-algebra structure of $(\mathcal{Z}_{\bullet},\partial'_{\bullet})$. Since $z^{'}_i$ and  $z_i^{''}$ do not involve $e_0$, it follows that $\partial_{i}(z^{'}_i)=0$ and $\partial_{i-1}(z^{''}_i)=0$ by the multiplication on $R[T_0]$, thus $\partial_i(z_i)=(-1)^{i+1}T_0z^{''}_i\neq 0$.
Therefore $\ker (\partial_i)=Z_i(a_1,\cdots,a_s)[T_0](-i)=Z'_i[T_0](-i)$.
\medskip

{\bf Step 2 $\Image (\partial_{i+1})$:} To compute the image, we consider an arbitrary element
 $$z_{i+1}\omega (T_0)=( \sum_{j} r_{j}e_{j_1}\wedge e_{j_2}\wedge ...\wedge e_{j_{i+1}})\omega (T_0)\in Z_{i+1}[T_0](-i-1)$$  and write
 $z_{i+1}=z'_{i+1}+z^{''}_{i+1}\wedge e_0$, where $z^{'}_{i+1}$ and $z_{i+1}^{''}$ do not involve  $e_0$. Recall  that $Z_{i+1}=Z_{i+1}(\ff)$.  Consider the Koszul complex $(K_{\bullet}(\ff),d_{\bullet})$. We have
$$0=d(z_{i+1})=d(z'_{i+1})+d(z^{''}_{i+1})\wedge e_0+(-1)^{i}bz^{''}_{i+1}.$$
 Since $z^{'}_{i+1}$ and  $z_{i+1}^{''}$ do not involve $e_0$, the above equation implies that $d(z^{''}_{i+1})=0$ and $bz^{''}_{i+1}=d((-1)^{i+1}z'_{i+1})$; so that $z^{''}_{i+1}\in (B'_{i}:_{Z'_i}b)$.
 
 Conversely, for any $z^{''}_{i+1}\in (B'_{i}:_{Z'_i}b)$, $bz^{''}_{i+1}=d((-1)^{i+1}z'_{i+1})$ for some $z'_{i+1}\in K'_{i+1}$, hence $z_{i+1}=(z'_{i+1}+z^{''}_{i+1}\wedge e_0)\in Z_{i+1}$.

On the other hand $\partial(z_{i+1})=\partial (z'_{i+1})+\partial( z^{''}_{i+1}\wedge e_0)=0+\partial( z^{''}_{i+1}\wedge e_0)=(-1)^{i}T_0z^{''}_{i+1}$.

It follows that $\Image (\partial_{i+1})=(T_0\widetilde{B_{i}}[T_0])(-i) $.  Finally the above two steps show that
$$H_i(\mathcal{D}_{\bullet})=\frac{(Z'_{i}\y_R R[T_0])(-i)}{(\widetilde{B_{i}}\y_R T_0R[T_0])(-i)}.$$

To see the last assertion, just notice that 
\begin{center}$((Z'_{i}[T_0]_{T_0})(-i))_{[0]}=Z'_{i}[T_0,T_{0}^{-1}]_{[-i]}=T_0^{-i}Z'_i$\quad  and \quad
$((\widetilde{B_{i}}T_0[T_0]_{T_0})(-i))_{[0]}=T_0^{-i}\widetilde{B_{i}}$.
\end{center}
 Therefore
\begin{equation}\label{emult}
(H_i(\mathcal{D}_{\bullet})_{T_0})_{[0]}\simeq Z'_i/\widetilde{B_{i}}=Z'_i/(B'_{i}:_{Z'_i}b)=H'_i/(0:_{H'_i}b)\simeq bH_i(a_1,\cdots,a_s).
\end{equation}
(ii). In the graded case every homomorphism in the above argument is homogenous except  the first and the last isomorphism in (\ref{emult}). To see the desired shift, notice that the bi-degree of $T_0$ is $(\deg(b),1)$. Hence considering $(H_i(\mathcal{D}_{\bullet})_{T_0})$ in bidegree $(0,0)$, we have $$(H_i(\mathcal{D}_{\bullet})_{T_0})_{[(0,0)]}\simeq \frac{T_0^{-i}((Z'_i)_{[i\deg(b)]})}{T_0^{-i}((\widetilde{B_{i}})_{[i\deg(b)]})}\simeq (\frac{Z'_i}{\widetilde{B_{i}}})_{[i\deg(b)]}. $$
Moreover there exists the following homogeneous exact sequence which yields the assertion
$$0\ra \widetilde{B_{i}} \ra Z'_i \ra bH_i(a_1,\cdots,a_s)(\deg(b))\ra 0. $$

 \end{proof}
Now, we are ready to explain our main Theorem in this section.
The dependence of the homology modules of $_0\mathcal{Z}^{+}_{\bullet}(\aa,\ff)$ to the generating sets becomes clear by this theorem, furthermore
it attracts attention to study the uniform annihilator of the  Koszul homology modules.
\begin{Theorem} \label{thr2} Let $R$ be a (Noetherian) ring, $\fa=(\aa)=(a_1,\cdots,a_s)$ and  $I=(\ff)=(b,a_1,\cdots,a_s)$.   Then
  $H_i(_0\mathcal{Z}^{+}_{\bullet}(\aa,\ff))\simeq bH_{i}(a_1,\cdots,a_s)$ for all $i\geq 1$ and $H_0(_0\mathcal{Z}^{+}_{\bullet}(\aa,\ff))\simeq R/(\fa:b)$; furthermore in the  graded  case  $H_i(_0\mathcal{Z}^{+}_{\bullet}(\aa,\ff))\simeq bH_{i}(a_1,\cdots,a_s)((i+1)\deg(b))$ for $i\geq 1$.
\end{Theorem}

\begin{proof}
We study the spectral sequences arising from the third quadrant double complex $C_{\frak{g}}^{\bullet}\y \mathcal{D}_{\bullet}$
with $C_{\frak{g}}^{0}\y D_{0}$ in the center. Since $T_0,T_1,...,T_s$ is a regular sequence  on $D_i$'s,  $(^{1}E_{ver})_{[0]}=H^{s+1}_{\frak{g}}( \mathcal{D}_{\bullet})_{[0]}$.
The latter is $\mathcal{Z}^{+}_{\bullet}$ by definition (\ref{complex1}) which is the truncated complex $_0\mathcal{Z}^{+}_{\bullet}(\aa,\ff)_{\geq 1}$.

On the other hand, in the horizontal spectral sequence we deal with $H^{j}_{\frak{g}}(H_i(\mathcal{D}_{\bullet}))$. $H_i(\mathcal{D}_{\bullet})$
is  annihilated by $(T_1,\cdots,T_s)$ by (\ref{Dconverge}) which implies that
$H^{j}_{\frak{g}}(H_i(\mathcal{D}_{\bullet}))=H^{j}_{(T_0)}(H_i(\mathcal{D}_{\bullet}))$ for all $i$ and $j$, also notice that
 $H_{(T_0)}^{0}(H_i(\mathcal{D}_{\bullet}))_{[0]}=0$ for $i\geq 1$ since $\beg(D_i)=i$.
Hence
\begin{equation}
(^2\E^{-i,-j}_{hor})_{[0]}=\left \{
\begin{array}{cc}
H_{(T_0)}^{0}(H_0(\mathcal{D}_{\bullet}))_{[0]} & i=0 \text{~~ and~~} j=0,  \\
H_{(T_0)}^{1}(H_i(\mathcal{D}_{\bullet}))_{[0]}=(H_i(\mathcal{D}_{\bullet})_{T_0})_{[0]}& i\geq 1 \text{ ~~and~~}  j=1, \\
 0& otherwise.
\end{array}
\right.
\end{equation}

Considering both  spectrals at a same time, we have 
\begin{eqnarray*}
\xymatrix{
(^{2}E_{hor})_{[0]}:&0&0&H_{(T_0)}^{0}(H_0(\mathcal{D}_{\bullet}))_{[0]}\\
\cdots&H_{(T_0)}^{1}(H_2(\mathcal{D}_{\bullet}))_{[0]}&H_{(T_0)}^{1}(H_1(\mathcal{D}_{\bullet}))_{[0]}&H_{(T_0)}^{1}(H_0(\mathcal{D}_{\bullet}))_{[0]}\\
\cdots &0&0&0\\
\mathcal{Z}^{+}_{\bullet}=(^{1}E_{ver})_{[0]}:&\vdots &\vdots & \vdots \\
H^{s+1}_{\frak{g}}(\mathcal{D}_{s+1})_{[0]}\ar@{--}[uuuurrr]\ar[r]&H^{s+1}_{\frak{g}}(\mathcal{D}_{s})_{[0]}\ar[r]&H^{s+1}_{\frak{g}}(\mathcal{D}_{s-1})_{[0]}\ar[r]&\cdots }
\end{eqnarray*}
By the convergence of the spectral sequences,  $H_{i+1}(_0\mathcal{Z}^{+}_{\bullet})=H_i(\mathcal{Z}^{+}_{\bullet})=H_{(T_0)}^{1}(H_{i+1}(\mathcal{D}_{\bullet}))_{[0]}$
for $i\geq 1$. Quite generally the latter is isomorphic to $(H_{i+1}(\mathcal{D}_{\bullet})_{T_0})_{[0]}$ which in turn is isomorphic to
$bH_{i+1}(a_1,\cdots,a_s)$ by Lemma \ref{lr2}. In the case where $i=0$ the convergence of the spectral sequences provides
 the following exact sequence
\begin{eqnarray}\label{exactseq}
\xymatrix{
0\ar[r] & H_{(T_0)}^{1}(H_1(\mathcal{D}_{\bullet}))_0\ar[r] & H_0(\mathcal{Z}^{+}_{\bullet})\ar^{\psi~~~~}[r] & H_{(T_0)}^{0}(H_0(\mathcal{D}_{\bullet}))_0 \ar[r] & 0.}
\end{eqnarray}
According to (\ref{complex1}), $H_0(\mathcal{Z}^{+}_{\bullet})=\coker(\phi_0)$ and $\psi$ is used to define $\tau_0$ in (\ref{tau}). Therefore $\Ker(\psi)=H_{1}(_0\mathcal{Z}^{+}_{\bullet})$ and its  image determines  $H_{0}(_0\mathcal{Z}^{+}_{\bullet})$ that we will compute below.

Since the horizontal spectral stabilizes in the second step, $^{\infty}E^{0,0}_{hor}=~~^{2}E^{0,0}_{hor}=H_{(T_0)}^{0}(H_0(\mathcal{D}_{\bullet}))$. Moreover
$\End(H^{s+1}_{\frak{g}}(\mathcal{D}_{s+1}))=0$ therefore $(^{\infty}E^{0,0}_{hor})_{[i]}=0$ for $i \geq 1$ which in particular implies that
$H_{(T_0)}^{0}(H_0(\mathcal{D}_{\bullet}))_{[1]}=0$. Specially we then have  $T_0H_{(T_0)}^{0}(H_0(\mathcal{D}_{\bullet}))_{[0]}=0$.
Recall that $H_0(\mathcal{D}_{\bullet})=Sym(I)/\ag Sym(I)$ and that we consider $ \mathcal{S}_I:=Sym_R(I)$ as an $S=R[T_0,\cdots,T_s]$-module  via the ring homomorphism
$S\rightarrow \mathcal{S}_I$ sending $T_0$ to  $b$  and $T_i$ to $a_i$ as an
element of  $(\mathcal{S}_I)_{[1]}=I$. Hence $T_0H_{(T_0)}^{0}(H_0(\mathcal{D}_{\bullet}))_{[0]}=0$ as an element in
$ (Sym(I)/\ag Sym(I))_{[1]}=I/\fa $
is equivalent to say that $H_{(T_0)}^{0}(H_0(\mathcal{D}_{\bullet}))_{[0]} \subseteq (\fa:b).$

 On the other hand, $H_{(T_0)}^{0}(H_0(\mathcal{D}_{\bullet}))_{[0]}=
(\fa:b)\cup ( \cup_{i=1}^{\infty} (\fa Sym_R^{i}(I):_RSym_R^{i+1}(I)))$, hence $H_{(T_0)}^{0}(H_0(\mathcal{D}_{\bullet}))_0\supseteq(\fa:b)$,
 which yields $H_{(T_0)}^{0}(H_0(\mathcal{D}_{\bullet}))_{[0]}= (\fa:b)$.

Further,  by  Lemma \ref{lr2}, $H_{(T_0)}^{1}(H_1(\mathcal{D}_{\bullet}))_{[0]}=(H_1(\mathcal{D}_{\bullet})_{T_0})_{[0]}=bH_1(f_1,f_2,...,f_s)$.
 Substituting these facts in the short exact sequence (\ref{exactseq}) the assertion follows.

\end{proof}
As a summary we have the following corollary
\begin{Corollary}\label{CZ+f0} Let $R$ be a (Noetherian) ring, $\fa=(\aa)=(a_1,\cdots,a_s)$ and  $I=(\ff)=(b,a_1,\cdots,a_s)$. If $bH_i(a_1,\cdots,a_s)=0$ for all $i\geq 1$, then there exists an
exact complex
\begin{eqnarray*}
\xymatrix{
0\ar[r] &\mathcal{Z}^{+}_{s}\ar[r] & \mathcal{Z}^{+}_{s-1}\ar[r] &\dots \ar[r] & \mathcal{Z}^{+}_0 \ar[r] & (\fa:_Rb)\ar[r]& 0}
\end{eqnarray*}
wherein     $\mathcal{Z}_i^{+}=\bigoplus_{j=i+1}^rZ_j(\ff)^{\oplus n_j}$ for some positive integers $n_j$ and $r$.
In particular the index of the lowest cycle which appears in the components increases along the complex.
\end{Corollary}

Theorem \ref{thr2} in conjunction with other conditions which imply the acyclicity of the complex $_0\mathcal{Z}^{+}_{\bullet}$ yields to non trivial  facts about the uniform annihilators of non-zero koszul homologies.
\begin{Corollary}\label{Ckosan}Let $(R,\fm)$ be a CM local ring of dimension $d$, $I=(b,a_1,\cdots,a_s)$  satisfies SD. If $\Dim(bH_i(a_1,\cdots,a_s))\leq d-s$ for all $i\geq 1$ then
\begin{itemize}
\item $bH_i(a_1,\cdots,a_s)=0$ for all $i\geq 1$ and
\item $\depth(R/((a_1,\cdots,a_s):I))\geq d-s$.
\end{itemize}
\end{Corollary}
\begin{proof} We consider the complex  $_0\mathcal{Z}^{+}_{\bullet}(\fa;(b,\fa))$. By Theorem \ref{thr2} the hypothesis on the dimensions implies that this complex is acyclic locally at codimension $s-1$. Then appealing  the acyclicity lemma, the sliding depth hypothesis shows that the complex is acyclic on the punctured spectrum and finally it is acyclic (the same technique which applied to prove Proposition \ref{t1}). Once more applying Theorem \ref{thr2}, we conclude the first assertion. The  depth inequality follows from the proof of Theorem \ref{t2}.
\end{proof}

As a final remark in this section, we notice that if $(\fa:I)$ is an $s$-residual intersection then $\Dim(IH_i(\fa))\leq d-s$ for all $i$. So that the latter property may be considered as a generalization of  algebraic residual Intersection.
\section{Uniform Annihilator of Koszul's Homologies}\label{sann}
Motivated by the results in the previous section, we concentrate on the uniform annihilator of Koszul homologies in this section. In the main theorem of this section we see that in any residual intersection $J=\fa:I$ the sliding depth condition passes from $I$ to $\fa$. This fact was known to experts only  in the presence of the $G_{\infty}$ condition \cite[1.8(3)+1.12]{U} which is not that surprising since $I$ is then generated by a $d$-sequence. However this was unknown even for prefect ideals of height $2$ which are not $G_{\infty}$, see \cite{EU} for such an example.

The next interesting result in this section is Corollary \ref{creskosan}, in which we show that for any residual intersection $J=\fa:I$ with $I$ satisfies sliding depth, $I$ is contained in the uniform annihilator of the non-zero Koszul homologies of $\fa$, hence a kind of universal properties for such ideals.

Next proposition is fundamental to prove the aforementioned  results though it is also interesting by itself.

\begin{Proposition} \label{uak} Let $(R,\fm)$ be a CM local ring of dimension $d$, $\fa=(f_1,f_2,...,f_s)$, $I=(f_1,f_2,...,f_r)$  be ideals with $s\leq r$, $J=(\fa:I)$ and $g=\Ht(I)=\Ht(\fa)\geq 1$. Let $n$ be an integer and suppose that $\Ht(J)\geq n \geq 1$ and that $\depth (H_{l}(f_1,f_2,..., f_{r}))\geq d-r+j$  for $l\geq r-n+2$. Then for any $k$ with  $r\geq k \geq s$ and any $j\geq k-n+2$,
\begin{itemize}
\item[i)] $\depth (H_{j}(f_1,f_2,..., f_{k})\geq d-k+j$; and
\item[ii)] $I \subseteq \Ann (H_{j}(f_1,f_2,..., f_{k}))$.
\end{itemize}
\textsf{Supplement,} both assertions hold for  sequences of the form $(f_1,\cdots,f_s,f_{i_1},\cdots,f_{i_{k-s}})$ where $\{f_{i_1},\cdots,f_{i_{k-s}}\}\subseteq \{f_{s+1},\cdots,f_r\}$.
\end{Proposition}

\begin{proof} The proof is by a {\bf recursive} induction on $k$. By assumption, the result is true for $k=r$. Suppose that it is true for $k\leq r$. We now apply a recursive induction on $j\geq (k-1)-n+2$ to show  the following claims
\begin{itemize}
\item[1)] $f_kH_j(f_1,f_2,...,f_{k-1})=0$;
\item[2)] $\depth (H_j(f_1,f_2,...,f_{k-1}))\geq d-(k-1)+j$;
\item[3)] $\depth (Z_j(f_1,f_2,...,f_{k-1}))\geq d-(k-1)+j$;
\item[4)] $\depth (B_j(f_1,f_2,...,f_{k-1}))\geq d-(k-1)+j$.
\end{itemize}
 Since these claims are clear for $j>k-g$, we  suppose that $j\geq (k-1)-n+3$ and that these claims hold for $j$. A simple depth chasing between cycles, boundaries and homologies shows that $(2)-(4)$, for $j-1$, follows from $(1)$ for $j-1$ and the induction hypotheses of $(2)-(4)$ for $j$. We just notice that condition $(1)$ provides the following exact sequence
 \begin{eqnarray*}
\xymatrix{
0\ar[r] &H_j(f_1,\cdots,f_{k-1})\ar[r] & H_{j}(f_1,\cdots,f_k)\ar[r] & H_{j-1}(f_1,\cdots,f_{k-1})\ar[r]& 0.}
\end{eqnarray*}

 We then proceed to prove $(1)$ for $j-1$ while assuming the four claims hold for $j$. To do this, we show that $f_kH_{j-1}(f_1,f_2,...,f_{k-1})_{\fp}=0$  by an induction on $\Ht(\fp)$. For  $\Ht(\fp)<n$ it is clear, since $I_{\fp}=a_{\fp}$. Now let $\fq$ be a prime ideal with $\Ht(\fq)\geq n$ and suppose that $(1)$ to $(4)$ hold  locally for all $\fp$ such that $\fp\subset \fq$. We may, and will, replace $(R,\fm)$ by $(R_{\fq},\fq R_{\fq})$. Considering the following exact sequence, out of the Koszul complex,
\begin{eqnarray*}
\xymatrix{
0\ar[r] &Z_j(f_1,\cdots,f_{k-1})\ar[r] & \bigwedge^j(R^{k-1})\ar[r] & B_{j-1}(f_1,\cdots,f_{k-1})\ar[r]& 0,}
\end{eqnarray*}
 we obtain  $\depth (B_{j-1})\geq \min\{\depth ( \bigwedge^j(R^{k-1})), \depth(Z_{j})-1\}\geq d-(k-1)+j-1$. By assumption, $j\geq k-n+2$, which implies that  $\depth (B_{j-1}(f_1,\cdots,f_{k-1}))\geq d-n+2\geq 2$. This time
\begin{eqnarray*}
\xymatrix{
0\ar[r] &B_{j-1}(f_1,\cdots,f_{k-1})\ar[r] & Z_{j-1}(f_1,\cdots,f_{k-1})\ar[r] & H_{j-1}(f_1,\cdots,f_{k-1})\ar[r]& 0}
\end{eqnarray*}
implies that $\depth (H_{j-1})\geq 1$, since $\depth (Z_{j-1})\geq 1$, as $\Dim(R_{\fq})\geq n\geq 1$.

We now consider  the exact sequence
\begin{eqnarray*}
\xymatrix{
0\ar[r] &H_{j}(f_1,\cdots,f_{k-1})\ar[r] & H_{j}(f_1,\cdots,f_{k-1},f_k)\ar[r] & \Gamma_{j-1}\ar[r]& 0}
\end{eqnarray*}

where $\Gamma_{j-1}=(0:_{H_{j-1}(f_1,\cdots,f_{k-1})}f_{k})$.

The induction hypotheses imply that
\begin{equation*}
\depth (\Gamma_{j-1})\geq \min\{\depth ( H_{j}(f_1,\cdots,f_{k-1},f_k)), \depth (H_{j}(f_1,\cdots,f_{k-1}))-1\}
\end{equation*}
\begin{equation*}
\geq \min\{d-k+j,d-(k-1)+j-1\} \geq d-k+j\geq d-k+(k-n+2)=d-n+2\geq 2.
 \end{equation*}

At last, we have the inclusion $\Gamma_{j-1}\subseteq H_{j-1}$ where the former has depth at least $2$ and the latter has depth at least $1$ while the equality holds on the punctured spectrum, by induction hypothesis, so that this inclusion must be an equality, which proves the claim$(1)$ for $j-1$.

 Applying the above argument for $k=r$ we have $\depth (H_j(f_1,\cdots,f_{r-1}))\geq d-(r-1)+j$ for all $j\geq (r-1)-n+2$ and that $f_rH_j(f_1,f_2,...,f_{r-1})=0$. However, as far as we fix $f_1,\cdots,f_s$ we may change the role of $f_r$ with any one of $\{f_{s+1},\cdots,f_r\}$. In other words, we have $\depth (H_j(f_1,\cdots,f_s,\cdots,\hat{f_i},\cdots,f_{r}))\geq d-(r-1)+j$ for all $j\geq (r-1)-n+2$ and that $f_iH_j(f_1,\cdots,f_s,\cdots,\hat{f_i},\cdots,f_{r})=0$.
Also notice that for any ideal $\fa\subseteq I'\subseteq I$, $\Ht(\fa:I')\geq \Ht(\fa:I)\geq n$. Therefore if the assertion $(i)$ holds for $(f_1,\cdots,f_s,f_{i_1},\cdots,f_{i_{k-s}})$ where $\{f_{i_1},\cdots,f_{i_{k-s}}\}\subseteq \{f_{s+1},\cdots,f_r\}$, then the above argument shows that  $(i)$ holds for $(f_1,\cdots,f_s,f_{i_1},\cdots,f_{i_{k-s-1}})$ and moreover $f_{i_{k-s}}H_j(f_1,\cdots,f_s,f_{i_1},\cdots,f_{i_{k-s-1}})=0$ for $j\geq k-1-n+2$. Since $f_{i_{k-s}}$ varies over the set $\{f_{s+1},\cdots,f_r\}$ we must have $IH_j(f_1,\cdots,f_s,f_{i_1},\cdots,f_{i_{k-s-1}})=0$ for $j\geq k-1-n+2$ which completes the proof.

\end{proof}

We next want to show, in Lemma \ref{sdc1}, that under the sliding depth condition, there exists  a similar  depth inequality for  cycles of the Koszul complex. For this, we need the following lemma which is the reminiscence  of the inductive Koszul's long exact sequence.
\begin{Lemma} \label{luck}
Let $R$ be a commutative ring, $\{f_0,f_1,...,f_r\}\subseteq R$, $Z_{\bullet}$, $B_{\bullet}$ the Koszul's cycles and boundaries  with respect to sequence $\{f_1,f_2,...,f_r\}$, $Z'_{\bullet}$ the Koszul's cycles with respect to sequence $(f_0,f_1,...,f_r)$ and $\Gamma_{\bullet}=(B_{\bullet}:_{Z_{\bullet}}f_0)$. Then for any  $ j\geq 0$, there exists an exact sequence
\begin{eqnarray*}
\xymatrix{
0\ar[r] &Z_{j}\ar[r] & Z'_{j}\ar[r] & \Gamma_{j-1}\ar[r]& 0}
\end{eqnarray*}
\end{Lemma}
\begin{proof}

The Koszul complex $K_{\bullet}(f_1,f_2,...,f_r)$ is canonically a sub complex of $K_{\bullet}(f_0,f_1,...,f_r)$. We will denote the differential of both complexes by $d$. Let  $\{e_0,e_1,...,e_r\}$ be the canonical basis of $R^{r+1}$.  Every  $\theta \in Z'_j$ can be written uniquely as a sum $\theta = e_0\wedge w+v$ where $e_0$ does not appear in terms  $w$ and $v$. Since  $\theta \in Z'_j$
\begin{equation*}
0=d(\theta)=d(e_0\wedge w+v)=f_0w-e_0\wedge d(w)+d(v).
\end{equation*}
It follows that
$$
\Biggl\{\begin{array}{c}
e_0\wedge d(w)=0 \\
f_0w=-d(v)
\end{array}
\text{~~hence~~}
\Biggl\{\begin{array}{c}
d(w)=0 \\
f_0w=-d(v)
\end{array} \text{~~therefore~~} w\in \Gamma_{j-1}. $$

We define $\varphi :Z'_j\rightarrow \Gamma_{j-1}$ to be  $\varphi (\theta)=\varphi (e_0\wedge w+v)=w$. The homomorphism $\varphi$ is well defined, as the expression is unique. It is onto, since  for any  $w \in \Gamma_{j-1}$ there is $v\in K_j(f_1,f_2,...,f_r)$ such that $f_0w=-d(v)$. Thence  $\theta=e_0\wedge w+v \in K_j(f_0,f_1,...,f_r)$ is the preimage of $w$.

 Moreover, $Z_j=\ker(\varphi)$. On one hand, if $\theta \in Z'_j$ belongs to $Z_j$, then there is no $e_0$ in its decomposition in terms of the canonical basis of $R^{r+1}$, consequently $\theta=e_0\wedge w+v=v$ and $w=0$. On the other hand, if $0=d(\theta)=d(e_0\wedge w+v)=w$, then $\theta=v$. Therefore $\theta \in Z'_j\cap K_{j}(f_1,f_2,...,f_r)=Z_{j}$.
\end{proof}

The next Lemma is needed in the proof of the main theorem \ref{tsd}. Notice that the sliding depth condition does not pass automatically  from homologies to cycles because there is a level restriction, see Proposition \ref{psdsdc}.

\begin{Lemma} \label{sdc1} Let $(R,\fm)$ be a CM local ring of dimension $d$, $\fa=(f_1,f_2,...,f_s)$ and $I=(f_1,f_2,...,f_r)$  be ideals with $s\leq r$. Let  $J=(\fa:I)$ and $g=\Ht(I)=\Ht(\fa)\geq 1$. Suppose that $ht(J)\geq n \geq 1$ and that $I$ satisfies $\SD$, then for $k\geq s$ and $j\geq k-n+2$, 
\begin{equation*}
\depth (Z_{j}(f_1,f_2,..., f_{k}))\geq d-k+j+1.
\end{equation*}
\end{Lemma}

\begin{proof}

By Proposition \ref{uak}, $I H_j(f_1,f_2,...,f_{k})=0$ whenever $s\leq k \leq r$ and $j\geq k-n+2$. Therefore by Lemma \ref{luck},  we have the following exacts sequences for $k-1\geq s$ and $j-1\geq (k-1)-n+2$
\begin{eqnarray*}
\xymatrix{
0\ar[r] &Z_{j}(f_1,f_2,...,f_{k-1})\ar[r] & Z_{j}(f_1,f_2,...,f_{k})\ar[r] & Z_{j-1}(f_1,f_2,...,f_{k-1})\ar[r]& 0.}
\end{eqnarray*}
 We will prove the result by a recursive induction on $k$. If $k=r$, then it follows from Proposition \ref{psdsdc}. Suppose that the result holds for $k\geq s+1$. Since $Z_{j}(f_1,f_2,...,f_{k-1})$ satisfies  $\SDC_1$ for $j\geq (k-1)-g+1$, we can use a new induction on $j$ to conclude that:
\begin{equation*}
\depth (Z_{j-1}(f_1,f_2,...,f_{k-1}))\geq d-(k-1)+(j-1) +1=d-k+j+1.
\end{equation*}
\end{proof}


We are now ready to prove the main Theorem of this section. As we mentioned at the beginning of the section, this fact was known just in the presence of the $G_{\infty}$ condition, under which the problem reduced to study the properties of ideals generated by  $d$-sequences. Here the only assumption is the sliding depth and that $\depth(R/I)\geq d-s$.

\begin{Theorem}\label{tsd} Let $(R,\fm)$ be a CM local ring of dimension $d$ and   $J=(\fa:I)$ be an $s$-residual intersection with $s\geq\Ht(I)\geq 1$. If $I$ satisfies $\SD$ and $\depth (R/I)\geq d-s$, then $\fa$ satisfies  $\SD$  as well.
\end{Theorem}
\begin{proof}
We apply  Lemma \ref{sdc1} for $n=s$, it then follows that
\begin{equation*}
\depth (Z_{j}(f_1,f_2,...,f_{s}))\geq d-s+j +1
\end{equation*}
for $j\geq 2$. The hypothesis $\depth (R/I)\geq d-s$ enables us to apply Corollary \ref{cunmixed} which in turn implies that $\depth (R/\fa)\geq d-s$; so that $\depth (\fa)\geq d-s+1$ and hence $\depth (Z_{1}(\fa)) \geq d-s +2$- that is  $\fa$ satisfies the $\SDC_1$ condition which is  equivalent to $\SD$ by Proposition \ref{psdsdc}.
\end{proof}
In the next proposition we prove a partial converse of Theorem \ref{tsd}.

\begin{Proposition}\label{psdinverse} Let $(R,\fm)$ be a CM local ring of dimension $d$, $\fa=(f_1,f_2,...,f_s)$ and  $I=(f_1,f_2,...,f_r)$ with $s\leq r$. Let $J=(\fa:I)$ be  an $s$-residual intersection. If $\fa$ satisfies  $\SD$, then  for any $s\leq k \leq r$ and  $j\geq k-s+1$,
\begin{itemize}
\item[(i)] $\depth (H_{j}(f_1,f_2,..., f_{k}))\geq d-k+j$ and
\item[(ii)] $\depth (Z_{j}(f_1,f_2,..., f_{k}))\geq d-k+j+1$.
\end{itemize}

 \end{Proposition}
 \begin{proof}

We will use induction on $k$. If $k=s$, then $(ii)$ is clear by Proposition \ref{psdsdc}. Suppose that the result holds for $k\geq s+1$. Then by induction hypothesis, we have
\begin{equation*}
\depth (H_{j}(f_1,f_2,..., f_{k}))\geq d-k+j\geq d-s+1
\end{equation*}
By a similar argument as in the proof of Proposition \ref{uak}, $IH_{j}(f_1,f_2,..., f_{k})=0$ for all $j\geq k-s+1$. Then we have the short exact sequences below provided by Lemma \ref{luck} for any $j\geq k+1-s+1$
 \begin{eqnarray*}
\xymatrix{
0\ar[r] &H_{j}(f_1,f_2,...,f_{k})\ar[r] & H_{j}(f_1,f_2,...,f_{k+1})\ar[r] & H_{j-1}(f_1,f_2,...,f_{k})\ar[r]& 0}
\end{eqnarray*}
and
\begin{eqnarray*}
\xymatrix{
0\ar[r] &Z_{j}(f_1,f_2,...,f_{k})\ar[r] & Z_{j}(f_1,f_2,...,f_{k+1})\ar[r] & Z_{j-1}(f_1,f_2,...,f_{k})\ar[r]& 0.}
\end{eqnarray*}

A depth chasing then completes the proof.

 \end{proof}
The next corollary of Theorem \ref{tsd} shows a tight relation between the uniform annihilator of non-zero Koszul homology modules and residual intersections.
\begin{Corollary}\label{creskosan} Let $(R,\fm)$ be a CM local ring of dimension $d$, $I$ satisfies $\SD$ and $\depth (R/I)\geq d-s$ with $s\geq\Ht(I)\geq 1$. Let $J=(\fa:I)$ be an $s$-residual intersection and use $H_j(\fa)$ to denote the $j$'th Koszul homology module with respect to a minimal generating set of $\fa$. Then
\begin{equation}\label{inclusion}\tag{$\star$}I\subseteq \bigcap_{\substack{j\geq 1}}\Ann(H_j(\fa)).\end{equation}
Furthermore, the equality happens in the following cases, if  $\depth(R/I)\geq d-s+1$ (hence $s\geq g+1$).
 \begin{itemize}
  \item{ If $\mu(I_{\fp})\leq s-i$ for a fix positive  integer $i$ and for all $\fp \in \Ass(R/I)$, then $I=\Ann(H_j(\fa))$ for all $1\leq j\leq i$.}
  \item{ If $R$ is Gorenstein and $\Ass(R/I)=\min(R/I)$ then $I= \bigcap_{\substack{j\geq 1}}\Ann(H_j(\fa))$.}
 \end{itemize}
\end{Corollary}
\begin{proof} By Corollary \ref{cunmixed}, $\fa$ is minimally generated by $s$ elements. According to Theorem \ref{tsd}, for any $j$, $\depth(H_j(\fa))\geq d-s+j$, in particular  $\Ht(\fp)\leq s-1$ for any $\fp \in \Ass(H_j(\fa))$ with $j\geq 1$ c.f. \cite[1.2.13]{BH}.

To show that $IH_j(\fa)=0$, we show that it has no associated prime. Notice that $IH_j(\fa)\subseteq H_j(\fa)$, hence for any $\fp \in \Ass(IH_j(\fa))$, $\Ht(\fp)\leq s-1$ for which $I_{\fp}=\fa_{\fp}$. Therefore $(IH_j(\fa))_{\fp}=0$ as desired.

To see the equality, we localize at associated primes of $I$ which have height at most $s-1$ so that $I_{\fp}=\fa_{\fp}$ for them. Now let  $(a_1,\cdots,a_s)$ be a minimal generating set of $\fa$. If $\mu(I_{\fp})\leq s-i$ then $H_j(a_1,\cdots,a_s)_{\fp}$ has a direct summand of $(R/\fa)_{\fp}$ for all $1\leq j\leq i$ and thus its annihilator is $\fa_{\fp}=I_{\fp}$. In the case where $R$ is Gorenstein and $\Ass(R/I)=\min(R/I)$, we see that locally at the associated primes of $R/I$, $I_{\fp}$ is unmixed and moreover $(a_1,\cdots,a_s)_{\fp}$ is not a regular sequence since $s>\Ht(\fp)$. Hence for $t=\Ht(I_{\fp})$, $\Ann(H_{s-t}(\fa_{\fp}))$ is the unmixed part of $\fa_{\fp}=I_{\fp}$ which is $I_{\fp}$.
\end{proof}

The next amazing example shows a benefit of relaxing the $G_s$ condition in our Theorems. However we took this example from \cite{EU} wherein Eisenbud and Ulrich used it to show the necessity  of the $G_s$ condition in their theorems.
\begin{Example} Let $(R,\fm)=\mathbb{Q}[x_1,\cdots,x_5]_{(x_1,\cdots,x_5)}$, $I=I_4(N)$ where
$$ N=
\left(\begin{array}{ccccclllll}
x_2&x_3&x_4&x_5\\x_1&x_2&x_3&x_4\\0&x_1&x_2&0\\0&0&x_1&x_2\\0&0&0&x_1
\end{array}\right).
$$
Let $\Delta_i$ be the minors obtained by omitting $i$'th row of $N$, we then consider $\fa$ to be the ideal generated by  entries of
$$\left(\begin{array}{ccccclllll}
\Delta_1&-\Delta_2&\Delta_3&-\Delta_4&\Delta_5
\end{array}\right)
\left(\begin{array}{ccccclllll}
x_1&0&0&0&0\\0&x_5&0&0&0\\x_3&x_4&x_5&x_3&0\\0&x_3&x_4&x_5&0\\0&x_2&x_3&x_4&x_5
\end{array}\right).
$$
$J=(\fa:I)$ is a $5$-residual intersection of $I$. $I$ is a perfect ideal of height $2$ hence it is SCM. Easily one can see that any prime that contains $I$ must contain $\fp=(x_1,x_2)$ that is $I$ is $\fp$-primary, since it is unmixed. The minimal presentation of $I$ is given by $N$; so that ${\rm Fitt}_4(I)=\fm$ as well ${\rm Fitt}_3(I)\not\subset\fp$ because $x_4^2-x_3x_5\not \in \fp$, ${\rm Fitt}_2(I)\subseteq\fp$ because any combination of $3$ "diagonal" elements contains either $0$ or $x_1$ or $x_2$. Thus $\mu(I_{\fp})=3$, in particular $I$ does not satisfy $G_3$. However it completely suits  in our Corollary \ref{creskosan} for $s=5$ and $i=2$, hence $\Ann(H_1(\fa))=\Ann(H_2(\fa))=I$. Also, since $\Ht(J)=5$, any prime ideal of height $2$ which contains $\fa$ must contain $I$ that means $I=\fa^{unm}$. Since the ring is Gorenstein, we then have $I=\fa^{unm}= \Ann(H_3(\fa))$ (our Macaulay2 system was not able to calculate $\Ann(H_1(\fa))$ directly!).
\end{Example}
We notice that Corollary  \ref{creskosan} presents a universal property for any ideal $\fa$ such that $J=(\fa:I)$ is an $s$-residual intersection of $I$. For instance in the above example for any other $\fa$ such that $J=(\fa:I)$ is a $5$-residual intersection of $I$, we have $\Ann(H_1(\fa))=\Ann(H_2(\fa))=\Ann(H_3(\fa))=\fa^{unm}=I$.

An unpublished result of G.Levin, \cite[Theorem 5.26]{V2} yields in  \cite[Corollary 2.7]{CHKV} that {\it If $R$ is a Noetherian ring and $\fa$ is an ideal of finite projective dimension $n$, then $(\Ann(H_1(\fa)))^{n+1}\subseteq \fa$. }  Combining this result with the above Corollary \ref{creskosan}, one faces the strange fact that $\fa$ and $I$ have a same radical if one adopts  the conditions of Corollary \ref{creskosan}. However the following simple example (taken from \cite{Hu}) shows that \cite[Corollary 2.7]{CHKV} is not correct, and thus the Levin's result is not true.
\begin{Example}Let $X= \left(\begin{array}{ccccclllll}
x_1&x_3&x_5&x_7\\x_2&x_4&x_6&x_8 \end{array}\right)$
 be a generic $2\times 4$ matrix,
$J=I_2(X)$, $\fa=(x_1x_4-x_2x_3,  x_1x_6-x_2x_5,  x_1x_8-x_2x_7)$ and $I=(x_1,x_2)$. Then
$J=\fa:I$ is a geometric $3-$residual intersection of $I$, $J$ is prime and  $\fa= I \cap J$ hence  $\fa$ is radical.
Furthermore $\Ann(H_1(\fa))=\fa^{unm}=I$ which is not contained in $\fa$

A simpler example would be the following degeneration of the above example. Let $\fa=(y_1^2-y_2^2,y_1y_3,y_2y_3)\in k[y_1,\cdots,y_4]$. Then $\Ann(H_1(\fa))=(y_1,y_2)\neq {\rm rad}(\fa)$. 
\end{Example}


At the end, we make a conjecture and ask a question. The conjecture will imply in turn  the Cohen-Macaulayness of algebraic residual intersections for ideals with sliding depth condition.
\begin{conj}\label{Conj1} If $(R,\fm)$ is a CM local ring of dimension $d$ and  $I$ satisfies $\SD$  with  $\depth (R/I)\geq d-s$. Then the disguised $s$-residual intersection  of $I$ coincides with the algebraic $s$-residual intersection.
\end{conj}

Here is a peculiar question  arises from Corollary \ref{creskosan}.
\begin{Question}\label{qinclusion} Does the equality hold in (\ref{inclusion}), if $s>\Ht(I)$?
\end{Question}

Finally we present, without proof (a proof is based on Proposition \ref{uak}), the following proposition which is  an effort toward Question \ref{qinclusion}.

 \begin{Proposition}\label{PinterannHi} Let $(R,\fm)$ be a CM local ring of dimension $d$, $\fa=(f_1,f_2,...,f_s)$, $I=(f_1,f_2,...,f_r)$ ideals with $s\leq r$ and  $\Ht(I)=\Ht(\fa)\geq 1$. Suppose that $\Ht(\fa:I)\geq n \geq 1$.
 \begin{itemize}
 \item If $s+1\leq r$ and $\depth (R/I)\geq d-n+1$, then $\Ann (H_j(f_1,f_2,...,f_r))=I$  for $0\leq j \leq r-s$.
 \item If $I$ satisfies  $\SD $, then for $ \mu \geq s $ and  $\mu-g\geq j\geq \mu -n +2$,
\begin{equation*}
\Ann (H_j(f_1,f_2,\cdots,f_{\mu}))= \bigcap_{i=0}^{\mu -s}\Ann (H_{j-i}(f_1,f_2,\cdots,f_{s}))
\end{equation*}

 \end{itemize}
  \end{Proposition}

\subsection*{Acknowledgment}
The  authors would like to thank Bernd Ulrich for encouragement and  illuminating questions during his stays in Brazil in 2013 and 2014. A part of this work was done while the first author visiting the university of Utah  supported by the PDE-fellowship, 233035/2014-1, from CNPq-Brazil, he expresses his gratitude to CNPq and thanks the university of Utah for the hospitality.  Thanks to E. Tavanfar, A. Boocher and S. Iyengar  for reading the draft and correcting some important typos. We also thank the referee for his/her careful reading of the manuscript and useful comments. 


\end{document}